\documentclass{emsprocart}
\input xy
\xyoption{all}

\contact[iyama@math.nagoya-u.ac.jp]{Osamu Iyama, Graduate School of Mathematics, Nagoya University, Chikusa-ku, Nagoya, 464-8602, Japan}

\newtheorem{theorem}{Theorem}[section]
\newtheorem{corollary}[theorem]{Corollary}
\newtheorem{lemma}[theorem]{Lemma}
\newtheorem{proposition}[theorem]{Proposition}

\theoremstyle{definition}
\newtheorem{definition}[theorem]{Definition}

\newtheorem{example}[theorem]{Example}
\newtheorem{examples}[theorem]{Examples}

\renewcommand{\mod}{\operatorname{mod}\nolimits}

\newcommand{\SL}{\operatorname{SL}\nolimits}
\newcommand{\GL}{\operatorname{GL}\nolimits}
\newcommand{\soc}{\operatorname{soc}\nolimits}

\newcommand{\add}{\operatorname{add}\nolimits}

\newcommand{\Hom}{\operatorname{Hom}\nolimits}

\newcommand{\CM}{\operatorname{CM}\nolimits}
\newcommand{\Sub}{\operatorname{Sub}\nolimits}
\newcommand{\fl}{\operatorname{f.l.}\nolimits}
\newcommand{\id}{\operatorname{id}\nolimits}
\newcommand{\pd}{\operatorname{pd}\nolimits}
\newcommand{\fd}{\operatorname{fd}\nolimits}
\newcommand{\Ker}{\operatorname{Ker}\nolimits}
\newcommand{\ann}{\operatorname{ann}\nolimits}
\newcommand{\depth}{\operatorname{depth}\nolimits}
\newcommand{\gl}{\operatorname{gl.dim}\nolimits}
\newcommand{\dom}{\operatorname{dom.dim}\nolimits}
\newcommand{\repdim}{\operatorname{rep.dim}\nolimits}
\newcommand{\Ext}{\operatorname{Ext}\nolimits}
\newcommand{\End}{\operatorname{End}\nolimits}
\newcommand{\Spec}{\operatorname{Spec}\nolimits}
\newcommand{\Tr}{\operatorname{Tr}\nolimits}
\newcommand{\ord}{\operatorname{ord}\nolimits}
\newcommand{\rank}{\operatorname{rank}\nolimits}

\newcommand{\Z}{{\mathbf{Z}}}

\newcommand{\A}{{\mathcal A}}
\newcommand{\CC}{{\mathcal C}}
\newcommand{\DD}{{\mathcal D}}
\newcommand{\KK}{{\mathcal K}}
\newcommand{\MM}{{\mathcal M}}

\newcommand{\PP}{{\mathcal P}}
\newcommand{\II}{{\mathcal I}}
\newcommand{\XX}{{\mathcal X}}

\newcommand{\TTT}{{\mathcal T}}

\title[Auslander-Reiten theory revisited]{Auslander-Reiten theory revisited}

\author[O. Iyama]{Osamu Iyama}
\begin{document}

\begin{abstract}
We recall several results in Auslander-Reiten theory for finite-dimensional algebras over fields and orders over complete local rings.
Then we introduce $n$-cluster tilting subcategories and higher theory of almost split sequences and Auslander algebras there.
Several examples are explained.
\end{abstract}


\maketitle

\section*{Introduction}\label{Introduction}

After three fundamental papers
\begin{itemize}
 \settowidth{\labelwidth}{ABr}
 \settowidth{\itemindent}{ \ }
\item[\cite{A-coherent}] M. Auslander, \emph{Coherent functors}, 1966,
\item[\cite{ABr}] M. Auslander and M. Bridger, \emph{Stable module theory}, 1969,
\item[\cite{A-repdim}] M. Auslander, \emph{Representation dimension of artin algebras}, 1971,
\end{itemize}
M. Auslander and I. Reiten established their theory in series of papers containing \cite{A-almost-I,A-almost-II,A-order,AR-stable,AR-almost-III,AR-almost-IV,AR-almost-V,AR-almost-VI}.
Their exciting application of abstract homological algebra opened a right way to understand the vast world of non-commutative algebras.
The aim of this article is not to give a survey of Auslander-Reiten theory.
We recall only several results from Auslander-Reiten theory focusing on the concept of Auslander algebras.
Then we give an introduction to a higher theory of Auslander algebras and almost split sequences as discussed in \cite{I-ar,I-aorder}.

Our starting point is `Auslander correspondence' given in \cite{A-repdim}, which is a bijection between representation-finite algebras and
algebras satisfying $\gl\Gamma\le 2\le\dom\Gamma$, called \emph{Auslander algebras}.
This was a milestone in modern representation theory of algebras leading to later Auslander-Reiten theory.
Since Auslander algebras have `representation theoretic realization' \cite{I-tau4} as endomorphism algebras of additive generators of module categories,
some aspects of Auslander-Reiten theory can be regarded as a generalization of properties of Auslander algebras to
general module categories without additive generators by using coherent functors \cite{A-coherent} and stable module categories \cite{ABr}.
Our hope is the following.

\medskip

\noindent
\textit{Certain classes of nice algebras should have `representation theoretic realization'
as endomorphism algebras of additive generators of certain classes of nice categories.
}

\medskip

 The article is divided into the following sections and subsections.
\medskip

\noindent
1. From Auslander-Reiten theory

 1.1. Auslander algebras

 1.2. Almost split sequences

 1.3. Representation dimension

\medskip

\noindent
2. $n$-Auslander-Reiten theory

2.1. $n$-cluster tilting subcategories

2.2. $n$-Auslander algebras

2.3. $n$-almost split sequences

2.4. Example: $n$-cluster tilting for $(n-1)$-Auslander algebras

2.5. Questions

\medskip

\noindent
3. Orders and Cohen-Macaulay modules

3.1. From Auslander-Reiten theory

3.2. $n$-Auslander-Reiten theory

3.3. Cohen-Macaulay modules and triangulated categories

3.4. $2$-cluster tilting for $2$-Calabi-Yau categories

\medskip

\noindent
4. Examples of $2$-cluster tilting objects

4.1. Preprojective algebras

4.2. Hypersurface singularities

\medskip

In Section \ref{n-Auslander algebras and n-almost split sequences} we consider finite-dimensional algebras satisfying $\gl\Gamma\le n+1\le\dom\Gamma$, called \emph{$n$-Auslander algebras}.
We introduce \emph{$n$-cluster tilting subcategories} (maximal $(n-1)$-orthogonal subcategories) of module categories.
We show that they give `representation theoretic realization' of $n$-Auslander algebras.
As a generalization of properties of $n$-Auslander algebras, we find a certain analogue of Auslander-Reiten theory for $n$-cluster tilting subcategories,
namely the $n$-Auslander-Reiten translation, the $n$-Auslander-Reiten duality and $n$-almost split sequences.

In Section \ref{Orders and Cohen-Macaulay modules} we study a large class of nice algebras, containing finite-dimensional algebras, called \emph{orders},
and their nice modules called \emph{Cohen-Macaulay modules}.
We recall results from Auslander-Reiten theory for the category of Cohen-Macaulay modules over orders \cite{A-order,A-isolated}.
The theory is particularly nice for Krull dimension two since there exist \emph{fundamental sequences}
connecting projective modules and injective modules and having properties similar to almost split sequences.
We also introduce $n$-Auslander-Reiten theory for $n$-cluster tilting subcategories of the category of Cohen-Macaulay modules over orders.
This is particularly nice for Krull dimension $n+1$ since there exist \emph{$n$-fundamental sequences} connecting projective modules and injective modules
and having properties similar to $n$-almost split sequences.
Consequently $n$-Auslander algebras for Krull dimension $n+1$ enjoy particularly nice properties, and they form $(n+1)$-Calabi-Yau algebras in certain cases.
We also compare them with non-commutative crepant resolutions introduced by Van den Bergh.
Although in this article we could not deal with an exciting aspect of $2$-cluster tilting theory on categorification of Fomin-Zelevinsky cluster algebras,
a few results will be explained in Subsection \ref{2-cluster tilting for 2-Calabi-Yau categories}.


We give several examples of $n$-cluster tilting subcategories.
In Subsection \ref{$n$-cluster tilting for higher Auslander algebras} we give an inductive construction of $(n-1)$-Auslander algebras with $n$-cluster tilting subcategories \cite{I-inductive}.
In Subsection \ref{Cluster tilting for preprojective algebras} we construct a family of finite-dimensional factor algebras of preprojective algebras parametrized by elements $w$ in the corresponding Coxeter group.
We show that the categories of their syzygy modules contain $2$-cluster tilting objects given by reduced expressions of $w$ \cite{BIRSc}.
In Subsection \ref{Hypersurface singularities} we give a necessary and sufficient condition for one-dimensional hypersurface singularities to have $2$-cluster tilting objects,
and classify $2$-cluster tilting objects \cite{BIKR} by applying results in previous sections and the theory of tilting mutation due to Riedtmann-Schofield and Happel-Unger.

\medskip
\noindent{\bf Conventions.}
Throughout this paper all modules are usually right modules.
For a noetherian ring $\Lambda$, we denote by $\mod\Lambda$ the category of finitely generated $\Lambda$-modules,
and by $\fl\Lambda$ the category of $\Lambda$-modules with finite length.
The composition $fg$ of morphisms $f$ and $g$ means first $g$, then $f$.
For a module $M\in\mod\Lambda$, we denote by $\add M$ the full subcategory of $\mod\Lambda$ consisting of direct summands of finite direct sums of copies of $M$.
For example, $\add\Lambda$ is the category of finitely generated projective $\Lambda$-modules.
We denote by $J_\Lambda$ the Jacobson radical of $\Lambda$.

Let $\XX$ be an additive category.
We call $\XX$ \emph{Krull-Schmidt} if any object in $\XX$ is isomorphic to a finite direct sum of objects whose endomorphism rings are local.
We call an object in $\XX$ \emph{basic} if all indecomposable direct summands are mutually non-isomorphic.
We say that $I$ is an \emph{ideal} of $\XX$ if we have a subgroup $I(X,Y)$ of $\Hom_{\XX}(X,Y)$ for any $X,Y\in\XX$ and
\[(Y,Z)\cdot I(X,Y)\cdot (W,X)\subset I(W,Z)\]
holds for any $W,X,Y,Z\in\XX$.
There exists an ideal $J_{\XX}$ called the \emph{Jacobson radical} of $\XX$ such that
$J_{\XX}(X,X)$ coincides with the Jaconson radical $J_{\End_{\XX}(X)}$ of the endomorphism ring $\End_{\XX}(X)$ for any $X\in\XX$.
For any object $M\in\XX$, we have an ideal $[M]$ of $\XX$ defined by
\begin{eqnarray}\label{ideal}
[M](X,Y):=\{f\in(X,Y)\ |\ f\ \mbox{ factors through $M^\ell$ for some }\ell\}.
\end{eqnarray}
Let $\XX$ be an extension-closed subcategory of an abelian category $\A$.
We denote by
\begin{eqnarray*}
\PP&:=&\{X\in\XX\ |\ \Ext^i_{\A}(X,\XX)=0 \mbox{ for } i>0\},\\
\II&:=&\{X\in\XX\ |\ \Ext^i_{\A}(\XX,X)=0 \mbox{ for } i>0\},
\end{eqnarray*}
the categories of projective objects and injective objects in $\XX$.
We say that $\XX$ has \emph{enough projectives} if, for any $X\in\XX$, there exists an exact sequence $0\to Y\to P\to X\to0$ in $\A$ with $Y\in\XX$ and $P\in\PP$.
Dually, we define \emph{enough injectives}.

%
%

\section{From Auslander-Reiten theory}\label{From Auslander-Reiten theory}

In this section we recall several results in Auslander-Reiten theory.

\subsection{Auslander algebras}\label{Auslander algebras}

A finite-dimensional algebra $\Lambda$ is called \emph{representation-finite} if there are only finitely many isoclasses of indecomposable $\Lambda$-modules.
In this case there exists $M\in\mod\Lambda$ such that $\add M=\mod\Lambda$,
and we have a finite-dimensional algebra $\End_\Lambda(M)$ with a categorical equivalence
\[\Hom_\Lambda(M,-):\mod\Lambda\to\add\End_\Lambda(M).\]
The representation theory of $\Lambda$ is encoded in the structure of the finite-dimensional algebra $\End_\Lambda(M)$.
To study $\End_\Lambda(M)$ we introduce a certain class of `Auslander-type conditions' on injective resolutions of noetherian rings (see \cite{AR-k-Gorenstein,HI}),
whose prototype is a property of commutative Gorenstein rings.
\begin{definition}\label{Auslander-type condition}
Let $\Gamma$ be a noetherian ring with a minimal injective resolution
\[0\to\Gamma\to I_0\to I_1\to\cdots\]
of the $\Gamma$-module $\Gamma$.
Let $n$ and $m$ be positive integers.
Following Tachikawa \cite{Ta} and Hoshino \cite{Ho}, we write
\[\dom\Gamma\ge n\]
if $I_i$ is a flat $\Gamma$-module for any $0\le i<n$.
More generally, we say that $\Gamma$ satisfies the \emph{$(m,n)$-condition} \cite{I-tau3,Hu} if $\fd I_i<m$ holds for any $0\le i<n$.
Notice that when $\Gamma$ is a finite-dimensional algebra, we have $\fd I_i=\pd I_i$.

While the dominant dimension is known to be left-right symmetric in the sense that $\dom\Gamma=\dom\Gamma^{\rm op}$ holds,
easy examples show that the $(m,n)$-condition is not left-right symmetric.
We say that $\Gamma$ satisfies the \emph{two-sided $(m,n)$-condition} if $\Gamma$ and $\Gamma^{\rm op}$ satisfies the $(m,n)$-condition.
\end{definition}
We call a finite-dimensional algebra $\Gamma$ an \emph{Auslander algebra} if
\[\gl\Gamma\le 2\le\dom\Gamma.\]
We have the following classical \emph{Auslander correspondence} \cite{A-repdim,ARS}. In Theorem \ref{higher auslander correspondence} we shall give a proof for a more general statement.
\begin{theorem}\label{auslander correspondence}
\begin{itemize}
\item[(a)] Let $\Lambda$ be a representation-finite finite-dimensional algebra and $\add M=\mod\Lambda$.
Then $\End_\Lambda(M)$ is an Auslander algebra.

\item[(b)] Any Auslander algebra can be obtained in this way.
This gives a bijection between the Morita-equivalence classes of representation-finite finite-dimensional algebras and those of Auslander algebras.
\end{itemize}
\end{theorem}
Auslander algebras $\Gamma$ enjoy a lot of interesting properties, which come from the representation theory of $\Lambda$.
Let us mention one of them.
\begin{definition}\label{Gorenstein condition}
Let $\Gamma$ be a noetherian ring with $\gl\Gamma=n<\infty$.
We say that $\Gamma$ satisfies the \emph{Gorenstein condition} if
any simple $\Gamma$-module $S$ satisfies
\begin{equation}\label{Gorenstein}
\Ext^i_\Gamma(S,\Gamma)=\left\{\begin{array}{cc}
0&i\neq n,\\
\mbox{a simple $\Gamma^{\rm op}$-module}&i=n,
\end{array}\right.
\end{equation}
and the same condition holds for the simple $\Gamma^{\rm op}$-modules.
More generally, we say that $\Gamma$ satisfies the \emph{restricted Gorenstein condition} if \eqref{Gorenstein} holds for any
simple $\Gamma$-module and $\Gamma^{\rm op}$-module $S$ with $\pd S=n$.
\end{definition}
The condition \eqref{Gorenstein} is satisfied by commutative local Gorenstein rings \cite{Ma},
and the Gorenstein condition is used in the definition of Artin-Schelter regular algebra \cite{ASc}.
The following observation \cite{I-tau3} shows the relationship between Gorenstein condition and Auslander-type condition.
\begin{proposition}\label{Gorenstein condition and (n,n)-condition}
A noetherian ring $\Gamma$ with $\gl\Gamma=n<\infty$ satisfies the restricted Gorenstein condition if and only if $\Gamma$ satisfies the two-sided $(n,n)$-condition.
\end{proposition}
In particular, any Auslander algebra $\Gamma$ satisfies the restricted Gorenstein condition.
In the next subsection we shall interpret this property in terms of the representation theory of the corresponding representation-finite algebra $\Lambda$.
It is the existence of almost split sequences.

\subsection{Almost split sequences}\label{Almost split sequences}

Let us recall the following observation \cite{ASS,ARS}.
\begin{proposition}\label{selfduality of almost split sequences}
Let $\Lambda$ be a finite-dimensional algebra, and let
\begin{equation}\label{almost split}
0\to Z\xrightarrow{g} Y\xrightarrow{f} X\to 0
\end{equation}
be an exact sequence in $\mod\Lambda$ with $f,g\in J_{\mod\Lambda}$.
Then
\[0\to(-,Z)\xrightarrow{g}(-,Y)\xrightarrow{f}J_{\mod\Lambda}(-,X)\to0\]
is exact if and only if
\[0\to(X,-)\xrightarrow{f}(Y,-)\xrightarrow{g}J_{\mod\Lambda}(Z,-)\to0\]
is exact.
In this case $X$, is indecomposable if and only if $Z$ is indecomposable.
\end{proposition}
When these conditions are satisfied, we call \eqref{almost split} an \emph{almost split sequence}.

Now we give a relationship between almost split sequences and the restricted Gorenstein conditions.
Assume that $\Lambda$ is a representation-finite finite-dimensional algebra with an additive generator $M$ of $\mod\Lambda$.
Let $\Gamma:=\End_\Lambda(M)$ be the corresponding Auslander algebra.
Then we have categorical equivalences
\begin{eqnarray*}
P_-:=(M,-)&:&\mod\Lambda\xrightarrow{\sim}\add\Gamma_\Gamma,\\
P^-:=(-,M)&:&\mod\Lambda\xrightarrow{\sim}\add{}_\Gamma\Gamma,
\end{eqnarray*}
such that $P^-\simeq\Hom_\Gamma(P_-,\Gamma)$.
We also have the functors
\begin{eqnarray*}
S_-:=(M,-)/J_{\XX}(M,-)&:&\mod\Lambda\to\mod\Gamma,\\
S^-:=(-,M)/J_{\XX}(-,M)&:&\mod\Lambda\to\mod\Gamma^{\rm op},
\end{eqnarray*}
whose images are semisimple $\Gamma$-modules and $\Gamma^{\rm op}$-modules.
For an almost split sequence \eqref{almost split}, we have projective resolutions
\begin{eqnarray*}
&0\to P_Z\xrightarrow{P_g}P_Y\xrightarrow{P_f}P_X\to S_X\to0,&\\
&0\to P^X\xrightarrow{P^f}P^Y\xrightarrow{P^g}P^Z\to S^Z\to0,&
\end{eqnarray*}
of a simple $\Gamma$-module $S_X$ and a simple $\Gamma^{\rm op}$-module $S^Z$. Applying $\Hom_\Gamma(-,\Gamma)$ to the first sequence and comparing with the second sequence, we have
\[\Ext^i_\Gamma(S_X,\Gamma)=\left\{\begin{array}{cc}
0&i\neq 2,\\
S^Z&i=2.
\end{array}\right.\]
This implies that $\Gamma$ satisfies the restricted Gorenstein condition.
We also observe that almost split sequences in $\mod\Lambda$ correspond to projective resolutions of simple $\Gamma$-modules.

\medskip
If a finite-dimensional algebra $\Lambda$ is not representation-finite, then one can not consider its Auslander algebra directly.
Instead, Auslander and Reiten dealt with the \emph{functor category} consisting of additive functors
from the category $\mod\Lambda$ to the category of abelian groups,
and developped an exciting duality theory in functor categories \cite{A-coherent,AR-stable}
leading to their theory of almost split sequences \cite{AR-almost-III,AR-almost-IV,AR-almost-V,AR-almost-VI,A-order}.

\medskip
Let us recall a construction from stable module theory \cite{ABr}.
Let $\Lambda$ be a noetherian ring in general. We have the functor
\[(-)^*:=\Hom_\Lambda(-,\Lambda):\mod\Lambda\leftrightarrow\mod\Lambda^{\rm op}\]
which induces a duality $(-)^*:\add\Lambda_\Lambda\stackrel{\sim}{\longleftrightarrow}\add{}_\Lambda\Lambda$.
For $X\in\mod\Lambda$, take a projective presentation
\[P_1\xrightarrow{g}P_0\stackrel{f}{\to}X\to0.\]
Define a $\Lambda^{\rm op}$-module $\Tr X$ by an exact sequence
\[P_0^*\xrightarrow{g^*}P_1^*\to\Tr X\to0.\]
We notice that $\Tr X$ depends on a choice of projective presentation of $X$.
To make $\Tr$ functorial, we need the \emph{stable category} of $\Lambda$ defined by
\[\underline{\mod}\Lambda:=(\mod\Lambda)/[\Lambda],\]
where the ideal $[\Lambda]$ of $\mod\Lambda$ is defined by \eqref{ideal}.
Then we have a fundamental duality
\[\Tr:\underline{\mod}\Lambda\stackrel{\sim}{\longleftrightarrow}\underline{\mod}\Lambda^{\rm op},\]
called the \emph{Auslander-Bridger transpose} \cite{ABr}.

\medskip
In the rest of this subsection let $\Lambda$ be a finite-dimensional algebra over a field $k$.
Then we have a duality
\[D:=\Hom_k(-,k):\mod\Lambda\stackrel{\sim}{\longleftrightarrow}\mod\Lambda^{\rm op}.\]
Define the \emph{costable category} of $\Lambda$ by
\[\overline{\mod}\Lambda:=(\mod\Lambda)/[D\Lambda],\]
where the ideal $[D\Lambda]$ of $\mod\Lambda$ is defined by \eqref{ideal}.
The duality $D:\mod\Lambda\leftrightarrow\mod\Lambda^{\rm op}$ induces a duality $D:\underline{\mod}\Lambda\leftrightarrow\overline{\mod}\Lambda^{\rm op}$,
and we have mutually quasi-inverse equivalences
\begin{eqnarray*}
\tau:=D\Tr&:&\underline{\mod}\Lambda\xrightarrow{\Tr}\underline{\mod}\Lambda^{\rm op}\xrightarrow{D}\overline{\mod}\Lambda,\\
\tau^-:=\Tr D&:&\overline{\mod}\Lambda\xrightarrow{D}\underline{\mod}\Lambda^{\rm op}\xrightarrow{\Tr}\underline{\mod}\Lambda,
\end{eqnarray*}
called the \emph{Auslander-Reiten translations}.
Especially, $\tau$ gives a bijection from isoclasses of indecomposable non-projective $\Lambda$-modules to
isoclasses of indecomposable non-injective $\Lambda$-modules.
We have the following functorial isomorphisms called the \emph{Auslander-Reiten duality} \cite{ASS,ARS}.
\begin{theorem}
There exist the following functorial isomorphisms for any modules $X,Y$ in $\mod\Lambda$
\[\underline{\Hom}_\Lambda(\tau^-Y,X)\simeq D\Ext^1_\Lambda(X,Y)\simeq\overline{\Hom}_\Lambda(Y,\tau X).\]
\end{theorem}
In particular, we have an isomorphism $\Ext^1_\Lambda(X,\tau X)\simeq D\underline{\End}_\Lambda(X)$ for any $X\in\mod\Lambda$, which gives an element of $\Ext^1_\Lambda(X,\tau X)$.
Thus we have the following existence theorem of almost split sequences \cite{ASS,ARS}.
\begin{theorem}
\begin{itemize}
\item[(a)] For any indecomposable non-projective module $X$ in $\mod\Lambda$, there exists an almost split sequence \eqref{almost split} such that $Z\simeq\tau X$.
\item[(b)] For any indecomposable non-injective module $X$ in $\mod\Lambda$, there exists an almost split sequence \eqref{almost split} such that $X\simeq\tau^-Z$.
\end{itemize}
\end{theorem}
This is fundamental in the study of the category $\mod\Lambda$.
For example, the structure theory of Auslander algebras using mesh categories of translation quivers was developped by
Riedtmann \cite{Ri}, Bongartz-Gabriel \cite{BG} and Igusa-Todorov \cite{IT1,IT2} (see also \cite{I-tau12-I,I-tau12-II,I-tau3}).

We end this subsection with the following exercise on further properties of Auslander algebras, which is closely related to
combinatorial characterizations of finite Auslander-Reiten quivers given by Brenner \cite{B} and Igusa-Todorov \cite{IT1,IT2} (see also \cite{I-tau4}).

\medskip
\noindent{\bf Exercise.}
Let $\Gamma$ be an Auslander algebra, and let $S$ be a simple $\Gamma$-module with $\pd S\le 1$.
Give a proof of the following statements by using the representation theory of the corresponding representation-finite algebra $\Lambda$.
\begin{itemize}
\item[(a)] Any composition factor $T$ of the $\Gamma^{\rm op}$-module $\Ext^1_\Gamma(S,\Gamma)$ satisfies $\pd T=2$.
\item[(b)] The socle $S'$ of the $\Gamma^{\rm op}$-module $\Hom_\Gamma(S,\Gamma)$ is simple and satisfies $\pd S'\le 1$.
Any composition factor $T$ of the $\Gamma^{\rm op}$-module $\Hom_\Gamma(S,\Gamma)/S'$ satisfies $\pd T=2$.
\end{itemize}
Also give a direct proof of the above statements by using homological algebra on $\Gamma$.

\subsection{Representation dimension}

For a representation-finite finite-dimensional algebra $\Lambda$, a `model' of the category $\mod\Lambda$ is given by the Auslander algebra $\Gamma$ of $\Lambda$.
In other words an Auslander algebra $\Gamma$ of $\Lambda$ has a `representation theoretic realization' $\mod\Lambda$.
It is natural to ask whether the correspondence $\Gamma\leftrightarrow\mod\Lambda$ can be generalized to other classes of algebras and categories.
One satisfactory answer will be given in Theorem \ref{higher auslander correspondence} which gives a bijection between `$n$-Auslander algebras'
and `$n$-cluster tilting subcategories'.
Let us recall another approach given by Auslander \cite{A-repdim}.
The following result \cite{A-repdim} due to Morita and Tachikawa deals with a wider class of algebras than considered in Theorem \ref{auslander correspondence}.
\begin{itemize}
\item Let $\Lambda$ be a finite-dimensional algebra and $M$ a generator-cogenerator in $\mod\Lambda$. Then $\Gamma:=\End_\Lambda(M)$ satisfies $\dom\Gamma\ge2$.
\item Any finite-dimensional algebra $\Gamma$ satsifying $\dom\Gamma\ge2$ is obtained in this way.
\end{itemize}
In view of this correspondence, Auslander defined the \emph{representation dimension} of a finite-dimensional algebra $\Lambda$ by
\[\repdim\Lambda:=\inf\{\gl\End_\Lambda(M)\ |\ M\ \mbox{ is a generator-cogenerator in } \mod\Lambda\}.\]
Then he proved that $\Lambda$ is representation-finite if and only if $\repdim\Lambda\le 2$.
In this sense the representation dimension measures how far an algebra is from being representation-finite.

A lot of interesting results are known recently (see \cite{X,EHIS}).
It was shown in \cite{I-repdim} that $\repdim\Lambda$ is always finite.
Igusa and Todorov \cite{IT3} proved that algebras with $\repdim\Lambda\le3$ have finite finitistic dimension.
Rouquier \cite{Rou} proved that the exterior algebra $\Lambda$ of an $n$-dimensional vector space satisfies $\repdim\Lambda=n+1$.
Oppermann \cite{O} gave an effective method to give a lower bound of representation dimension.

\section{$n$-Auslander-Reiten theory}\label{n-Auslander algebras and n-almost split sequences}

\subsection{$n$-cluster tilting subcategories}\label{n-cluster tilting subcategories}

The concept of functorially finite subcategories was introduced by Auslander and Smal\o\ \cite{ASm1,ASm2} and studied mainly by the school of Auslander.
Today it turns out to be one of the fundamental concepts in representation theory of algebras, especially tilting theory \cite{AR-contravariant,ABu}.

Let $\XX$ be an additive category and $\CC$ a subcategory of $\XX$.
We call $\CC$ \emph{contravariantly finite} if, for any $X\in\XX$, there exists a morphism $f\in(C,X)$ with $C\in\CC$ such that
\[(-,C)\xrightarrow{f}(-,X)\to0\]
is exact. Such $f$ is called a \emph{right $\CC$-approximation} of $X$.
Dually, a \emph{covariantly finite} subcategory and a \emph{left $\CC$-approximation} are defined.
A contravariantly and covariantly finite subcategory is called \emph{functorially finite}.

A class of functorially finite subcategories is introduced in \cite{I-ar,I-aorder} (see also \cite{KR}).
\begin{definition}\label{definition of n-cluster tilting}
Let $\XX$ be an extension-closed subcategory of an abelian category $\A$.
We call a subcategory $\CC$ of $\XX$ \emph{$n$-rigid} if
\[\Ext^i_{\A}(\CC,\CC)=0\]
for any $0<i<n$. We call a subcategory $\CC$ of $\XX$ \emph{$n$-cluster tilting} (or \emph{maximal $(n-1)$-orthogonal}) if it is functorially finite and
\begin{eqnarray*}
\CC&=&\{X\in\XX\ |\ \Ext^i_{\A}(\CC,X)=0 \mbox{ for } 0<i<n\}\\
&=&\{X\in\XX\ |\ \Ext^i_{\A}(X,\CC)=0 \mbox{ for } 0<i<n\}.
\end{eqnarray*}
An object $M\in\XX$ is called \emph{$n$-cluster tilting} (respectively, \emph{$n$-rigid}) if so is $\add M$.
\end{definition}
Clearly $\XX$ is a unique $1$-cluster tilting subcategory of $\XX$, and $2$-cluster tilting subcategories are often called \emph{cluster tilting}.
If an $n$-cluster tilting subcategory $\CC$ is contained in an $n$-rigid subcategory $\DD$, then we have $\CC=\DD$.

Let us give a few properties of $n$-cluster tilting subcategories \cite{I-ar,I-aorder}.
\begin{proposition}\label{one-side condition}
Assume that an additive category $\XX$ has enough projectives and enough injectives.
Let $\CC$ be a functorially finite subcategory of $\XX$.
Then the following conditions are equivalent.
\begin{itemize}
\item[(a)] $\CC$ is an $n$-cluster tilting subcategory of $\XX$.
\item[(b)] $\CC=\{X\in\XX\ |\ \Ext^i_{\A}(\CC,X)=0 \mbox{ for } 0<i<n\}$ and $\CC$ contains all projective objects in $\XX$.
\item[(c)] $\CC=\{X\in\XX\ |\ \Ext^i_{\A}(X,\CC)=0 \mbox{ for } 0<i<n\}$ and $\CC$ contains all injective objects in $\XX$.
\end{itemize}
\end{proposition}
The following result shows that any object in $\XX$ has a finite sequence of right (respectively, left) $\CC$-approximations.
There is a nice interpretation of this property in terms of relative homological algebra of Auslander-Solberg \cite{ASo-I,ASo-II,ASo-III}. See also \cite{L}.
\begin{proposition}\label{M-approximation}
Assume that $\XX$ has enough projectives and enough injectives.
Let $\CC$ be an $n$-cluster tilting subcategory of $\XX$.
For any $X\in\XX$, there exist exact sequences
\begin{eqnarray*}
&0\to C_{n-1}\xrightarrow{f_{n-1}}\cdots\xrightarrow{f_2}C_1\xrightarrow{f_1}C_0\xrightarrow{f_0}X\to0,&\\
&0\to X\xrightarrow{f'_0}C'_0\xrightarrow{f'_1}C'_1\xrightarrow{f'_2}\cdots\xrightarrow{f'_{n-2}}C'_{n-1}\to0,&
\end{eqnarray*}
in $\A$ with terms in $\CC$ such that the following sequences are exact on $\CC$
\begin{eqnarray*}
&0\to(-,C_{n-1})\xrightarrow{f_{n-1}}\cdots\xrightarrow{f_2}(-,C_1)\xrightarrow{f_1}(-,C_0)\xrightarrow{f_0}(-,X)\to0,&\\
&0\to(X,-)\xrightarrow{f'_0}(C'_0,-)\xrightarrow{f'_1}(C'_1,-)\xrightarrow{f'_2}\cdots\xrightarrow{f'_{n-2}}(C'_{n-1},-)\to0.&
\end{eqnarray*}
\end{proposition}
In the rest of this subsection we give a few examples.

\begin{examples} \label{examples-A}
(a) Let $\Lambda_n$ be a finite-dimensional algebra given by the quiver
\[0\xrightarrow{a_0}1\xrightarrow{a_1}2\xrightarrow{a_2}\cdots\xrightarrow{a_{n-2}}n-1\xrightarrow{a_{n-1}}n\]
with relations $a_ia_{i+1}=0$ for any $0\le i\le n-2$.
Then the Auslander-Reiten quiver of $\Lambda_n$ is the following, where $P_i$ and $S_i$ are projective and simple modules associated with the vertex $i$, respectively.
\[\xymatrix@C0.2cm@R0.2cm{
&P_{n-1}\ar[rd]&&P_{n-2}\ar[rd]&&\cdots&&P_1\ar[rd]&&P_0\ar[rd]\\
P_n=S_n\ar[ru]&&S_{n-1}\ar[ru]\ar@{..}[ll]&&S_{n-2}\ar@{..}[ll]&\cdots&S_2\ar[ru]&&S_1\ar[ru]\ar@{..}[ll]&&S_0\ar@{..}[ll]}
\]
One can easily check that $\add(\Lambda_n\oplus S_0)$ is a unique $n$-cluster tilting subcategory of $\mod\Lambda_n$
such that $\End_{\Lambda_n}(\Lambda_n\oplus S_0)\simeq\Lambda_{n+1}$.

(b) Let $\Lambda$ be a representation-finite selfinjective algebra such that the tree type of the stable Auslander-Reiten quiver of $\Lambda$
is either $A_m$, $B_m$, $C_m$ or $D_m$.
Then a combinatorial criterion for existence of $n$-cluster tilting subcategories in $\mod\Lambda$ is given in \cite{I-ar}.
This is related to the example (d) below.

(c) Another example is given by Geiss, Leclerc and Schr\"oer \cite{GLS1,GLS2}.
See Subsection \ref{Cluster tilting for preprojective algebras} for the definition of preprojective algebras.
\end{examples}
\begin{theorem}
Let $\Lambda$ be a preprojective algebra of Dynkin type. Then there exists a $2$-cluster tilting object in $\mod\Lambda$.
\end{theorem}
In fact, each reduced expression of the longest element in the corresponding Weyl group gives a 2-cluster tilting object in $\mod\Lambda$.

We shall explain this in Theorem \ref{cluster tilting from reduced expression} below, in a general setting.

(d) An interesting example of categories with $n$-cluster tilting subcategories is given by \emph{$(n-1)$-cluster categories} (\cite{BMRRT,CCS} for $n=2$, and \cite{Th,Z,BM,W} for $n\ge2$),
which are $n$-Calabi-Yau triangulated categories constructed from the derived categories of path algebras \cite{Ke-orbit}.
(Notice that our $n$-cluster tilting objects are sometimes called `$(n-1)$-(cluster) tilting' in this setting.)

\subsection{$n$-Auslander algebras}

Throughout this subsection we fix a positive integer $n$.
We call a finite-dimensional algebra $\Gamma$ an \emph{$n$-Auslander algebra} if
\begin{equation}\label{n-Auslander algebra}
\gl\Gamma\le n+1\le\dom\Gamma.
\end{equation}
It is easily checked that any $n$-Auslander algebra $\Gamma$ is either semisimple or satisfies $\gl\Gamma=n+1=\dom\Gamma$.
We have the following \emph{$n$-Auslander correspondence} \cite{I-aorder},
where we call an additive category \emph{$n$-cluster tilting} if it is equivalent to an $n$-cluster tilting subcategory of $\mod\Lambda$ for some finite-dimensional algebra $\Lambda$.
\begin{theorem}\label{higher auslander correspondence}
\begin{itemize}
\item[(a)] Let $\Lambda$ be a finite-dimensional algebra and $M$ an $n$-cluster tilting object in $\mod\Lambda$. Then $\End_\Lambda(M)$ is an $n$-Auslander algebra.
\item[(b)] Any $n$-Auslander algebra $\Gamma$ can be obtained in this way.
This gives a bijection between the sets of equivalence classes of $n$-cluster tilting categories with additive generators and Morita-equivalence classes of $n$-Auslander algebras.
\end{itemize}
\end{theorem}

\begin{proof}
Although this is shown in \cite{I-aorder} in a general setting, we give a complete proof for our setting here for the convenience of the reader.

(a) We have an equivalence
\[P_-=\Hom_\Lambda(M,-):\add M_\Lambda\xrightarrow{\sim}\add\Gamma_\Gamma.\]
For any $X\in\mod\Gamma$, we take a projective resolution
$P_1\xrightarrow{f}P_0\to X\to0$.
Using the equivalence $P_-$, there exists a morphism
$g\in\Hom_\Lambda(M_1,M_0)$
in $\add M$ such that $f=P_{g}$.
Applying Proposition \ref{M-approximation} to $\Ker g$, we have an exact sequence
\[0\to M_{n+1}\to\cdots\to M_2\to\Ker g\to0\]
such that
\[0\to P_{M_{n+1}}\to\cdots\to P_{M_2}\to P_{\Ker g}\to0\]
is exact.
Since $P_{\Ker g}=\Ker f$, we have $\pd X\le n+1$, and so $\gl\Gamma\le n+1$.

Now we shall show $\dom\Gamma\ge n+1$.
Take an injective resolution
\[0\to M\to I_0\to I_1\to\cdots\to I_n\]
of the $\Lambda$-module $X$. Applying $P_-$, we have an exact sequence
\[0\to\Gamma\to P_{I_0}\to P_{I_1}\to\cdots\to P_{I_n}\]
since $M$ is $n$-rigid.
Since $D\Lambda\in\add M$, we have that $P_{I_i}$ is a projective $\Gamma$-module.
Moreover $P_{D\Lambda}=DM=D\Hom_\Lambda(\Lambda,M)$ holds.
Since $\Lambda\in\add M$, we have that $P_{I_i}$ is an injective $\Gamma$-module.
Thus we have $\dom\Gamma\ge n+1$.

(b) Let $I$ be a direct sum of all indecomposable projective-injective $\Gamma$-modules.
Define a projective $\Gamma$-module by $Q:=\Hom_{\Gamma^{\rm op}}(DI,\Gamma)$.
Put $\Lambda:=\End_\Gamma(Q)$ and $M:=\Hom_\Gamma(Q,\Gamma)$.
We have the functors
\begin{eqnarray*}
P_-:=\Hom_\Lambda(M,-)&:&\mod\Lambda\to\mod\Gamma,\\
Q_-:=\Hom_\Gamma(Q,-)&:&\mod\Gamma\to\mod\Lambda,
\end{eqnarray*}
such that $Q_{P_-}\simeq\id_{\mod\Lambda}$.
It is easily checked that we have mutually quasi-inverse equivalences
\begin{equation}\label{P and Q}
P_-:\add(D\Lambda)_\Lambda\to\add I_\Gamma\ \mbox{ and }\ Q_-:\add I_\Gamma\to\add(D\Lambda)_\Lambda.
\end{equation}

Take a minimal injective resolution
\begin{equation}\label{Gamma injective}
0\to\Gamma\to I_0\to\cdots\to I_n
\end{equation}
of the $\Gamma$-module $\Gamma$. Applying $Q_-$, we have an exact sequence
\begin{equation}\label{M injective}
0\to M\to Q_{I_0}\to\cdots\to Q_{I_n}
\end{equation}
with injective $\Lambda$-modules $Q_{I_i}$.
Applying $P_-$ to \eqref{M injective} and using \eqref{P and Q}, we get the exact sequence \eqref{Gamma injective}.
This implies that $M$ is an $n$-rigid $\Lambda$-module, and that $\End_\Lambda(M)\simeq\Gamma$.

We shall show that $M$ is an $n$-cluster tilting object in $\mod\Lambda$.
Since $M$ is a generator in $\mod\Lambda$, it is enough, by Proposition \ref{one-side condition}, to show that
any $X\in\mod\Lambda$ satisfying $\Ext^i_\Lambda(M,X)=0$ for any $0<i<n$ belongs to $\add M$.
Take an injective resolution
\[0\to X\to I'_0\to\cdots\to I'_n.\]
Applying $P_-$, we have an exact sequence
\[0\to P_X\to P_{I'_0}\to\cdots\to P_{I'_n}\]
with projective $\Gamma$-modules $P_{I'_i}$ since we asumed $\Ext^i_\Lambda(M,X)=0$ for any $0<i<n$.
Since $\gl\Gamma\le n+1$, we have $P_X\in\add\Gamma$.
Thus $X=Q_{P_X}\in\add Q_\Gamma=\add M$.
\end{proof}
We also have the relative version \cite{I-tau4,I-aorder} of Theorem \ref{higher auslander correspondence}.
For a tilting $\Lambda$-module $T$ with $\pd T<\infty$ \cite{Ha,Mi}, define a full subcategory $T^\perp$ of $\mod\Lambda$ by
\[T^\perp:=\{X\in\mod\Lambda\ |\ \Ext^i_\Lambda(T,X)=0 \mbox{ for } i > 0\}.\]
Then $T^\perp$ forms an extension-closed subcategory of $\mod\Lambda$ with enough projectives $\add T$ and enough injectives $\add D\Lambda$.

Let $m$ and $n$ be integers with $0\le m\le n$ and $1\le n$.
We call an additive category \emph{$m$-relative $n$-cluster tilting} if it is equivalent to an $n$-cluster tilting subcategory of $T^\perp$
for some finite-dimensional algebra $\Lambda$ and some tilting $\Lambda$-module $T$ with $\pd T\le m$.
On the other hand, we say that a finite-dimensional algebra is an \emph{$m$-relative $n$-Auslander algebra} if $\Gamma$ satisfies $\gl\Gamma\le n+1$ and the two-sided $(m+1,n+1)$-condition in Definition \ref{Auslander-type condition}.
\begin{theorem}
Let $m$ and $n$ be integers with $0\le m\le n$ and $1\le n$.
\begin{itemize}
\item[(a)] Let $\Lambda$ be a finite-dimensional algebra and $T$ a tilting $\Lambda$-module with $\pd T\le m$. Let $M$ be an $n$-cluster tilting object in $T^\perp$.
Then $\End_\Lambda(M)$ is an $m$-relative $n$-Auslander algebra.
\item[(b)] Any $m$-relative $n$-Auslander algebra is obtained in this way.
This gives a bijection between the sets of equivalence classes of $m$-relative $n$-cluster tilting categories with additive generators and Morita-equivalence classes of $m$-relative $n$-Auslander algebras.
\end{itemize}
\end{theorem}
By Proposition \ref{Gorenstein condition and (n,n)-condition}, any $m$-relative $n$-Auslander algebra $\Gamma$ with $\gl\Gamma=n+1$ satisfies the Gorenstein condition.
From the viewpoint of Subsection \ref{Almost split sequences}, it is natural to consider an analogous theory of almost split sequences for $n$-cluster tilting subcategories, which is the subject in the next subsection.

\subsection{$n$-almost split sequences}
In this subsection we introduce an analogous theory for $n$-cluster tilting subcategories as disscussed in the paper \cite{I-ar}.

Let $\Lambda$ be a finite-dimensional algebra, and let $\CC$ be an $n$-cluster tilting subcategory of $\mod\Lambda$ with $n\ge1$.
Let
\[\tau=D\Tr:\underline{\mod}\Lambda\to\overline{\mod}\Lambda\ \mbox{ and }\ \tau^-=\Tr D:\overline{\mod}\Lambda\to\underline{\mod}\Lambda\]
be the Auslander-Reiten translations defined in Subsection \ref{Almost split sequences}.
We denote by
\[\Omega:\underline{\mod}\Lambda\to\underline{\mod}\Lambda\ \mbox{ and }\ \Omega^-:\overline{\mod}\Lambda\to\overline{\mod}\Lambda\]
the syzygy and the cosyzygy functors.
Consider the functors
\begin{eqnarray*}
\tau_n:=\tau\Omega^{n-1}&:&\underline{\mod}\Lambda\xrightarrow{\Omega^{n-1}}\underline{\mod}\Lambda\xrightarrow{\tau}\overline{\mod}\Lambda,\\
\tau_n^-:=\tau^-\Omega^{-(n-1)}&:&\overline{\mod}\Lambda\xrightarrow{\Omega^{-(n-1)}}\overline{\mod}\Lambda\xrightarrow{\tau^-}\underline{\mod}\Lambda.
\end{eqnarray*}
These functors $\tau_n$ and $\tau_n^-$ are not equivalences in general though $(\tau_n^-,\tau_n)$ forms an adjoint pair.
We denote by $\underline{\CC}$ (respectively, $\overline{\CC}$)
the corresponding subcategory of $\underline{\mod}\Lambda$ (respectively, $\overline{\mod}\Lambda$).
The following result shows that it is fundamental in the study of $n$-cluster tilting subcategories.
\begin{theorem}\label{n-AR translation}
The functors $\tau_n$ and $\tau_n^-$ induce mutually quasi-inverse equivalences
\[\tau_n:\underline{\CC}\to\overline{\CC}\ \mbox{ and }\ \tau_n^-:\overline{\CC}\to\underline{\CC}.\]
\end{theorem}
Especially, $\tau_n$ gives a bijection from isoclasses of indecomposable non-projective objects in $\CC$ to
isoclasses of indecomposable non-injective objects in $\CC$.
We call $\tau_n$ and $\tau_n^-$ the \emph{$n$-Auslander-Reiten translations}.
For $n$-cluster tilting subcategories $\CC$, we have the following \emph{$n$-Auslander-Reiten duality} which is analogous to the Auslander-Reiten duality.
\begin{theorem}
There exist the following functorial isomorphisms for any modules $X,Y$ in $\CC$
\[\underline{\Hom}_\Lambda(\tau_n^-Y,X)\simeq D\Ext^n_\Lambda(X,Y)\simeq\overline{\Hom}_\Lambda(Y,\tau_nX).\]
\end{theorem}
In particular, we have an isomorphism $\Ext^n_\Lambda(X,\tau_nX)\simeq D\underline{\End}_\Lambda(X)$ for any $X\in\CC$.
Thus one can get a long exact sequence given by a class of $\Ext^n_\Lambda(X,\tau_nX)$ as in the case of Auslander-Reiten theory.
Let us start with the following preliminary observation which is analogous to Proposition \ref{selfduality of almost split sequences}.
\begin{proposition}
Let
\begin{equation}\label{n-almost split}
0\to Y\xrightarrow{f_{n+1}}C_n\xrightarrow{f_n}\cdots\xrightarrow{f_2}C_1\xrightarrow{f_1}X\to0
\end{equation}
be an exact sequence with terms in $\CC$ and $f_i\in J_{\CC}$ for any $i$.
Then
\[0\to(-,Y)\xrightarrow{f_{n+1}}(-,C_n)\xrightarrow{f_n}\cdots\xrightarrow{f_2}(-,C_1)\xrightarrow{f_1}J_{\CC}(-,X)\to0\]
is exact on $\CC$ if and only if
\[0\to(X,-)\xrightarrow{f_1}(C_1,-)\xrightarrow{f_2}\cdots\xrightarrow{f_n}(C_n,-)\xrightarrow{f_{n+1}}J_{\CC}(Y,-)\to0\]
is exact on $\CC$.
In this case, $X$ is indecomposable if and only if $Y$ is indecomposable.
\end{proposition}
When these conditions are satisfied, the sequence \eqref{n-almost split} is called an \emph{$n$-almost split sequence}.
We have the following existence theorem of $n$-almost split sequences.
\begin{theorem}
\begin{itemize}
\item[(a)] For any indecomposable non-projective module $X$ in $\CC$, there exists an $n$-almost split sequence \eqref{n-almost split} such that $Y\simeq\tau_nX$.
\item[(b)] For any indecomposable non-injective module $X$ in $\CC$, there exists an $n$-almost split sequence \eqref{n-almost split} such that $X\simeq\tau_n^-Y$.
\end{itemize}
\end{theorem}
It is interesting to have a structure theory of $n$-Auslander algebras which is similar to the structure theory of Auslander algebras using translation quivers and mesh relations \cite{Ri,BG,IT1,IT2,I-tau12-I,I-tau12-II,I-tau3}.
For $n=2$, one can hope that it is given by some generalization of Jacobian algebras of quivers with potentials \cite{DWZ} like 3-CY algebras \cite{CRo,G,Ke},
from the viewpoint of Proposition \ref{(d-1) and d} below.

We end this subsection with the following useful characterization \cite{I-aorder} of $n$-cluster tilting objects, which we shall use later.
\begin{lemma}\label{criterion for n-cluster tilting}
Let $\Lambda$ be a finite-dimensional algebra, and $T$ a tilting $\Lambda$-module with $\pd T\le n+2$.
Let $M$ be an $n$-rigid generator-cogenerator in $T^\perp$.
Then the following conditions are equivalent.
\begin{itemize}
\item[(a)] $M$ is an $n$-cluster tilting object in $T^\perp$.
\item[(b)] $\gl\End_\Lambda(M)\le n+1$.
\item[(c)] For any indecomposable object $X\in\add M$, there exists an exact sequence
\[0\to C_{n+1}\xrightarrow{f_{n+1}}C_n\xrightarrow{f_n}\cdots\xrightarrow{f_2}C_1\xrightarrow{f_1}X\]
with terms in $\add M$ such that the following sequence is exact
\[0\to(M,C_{n+1})\xrightarrow{f_{n+1}}(M,C_n)\xrightarrow{f_n}\cdots\xrightarrow{f_2}(M,C_1)\xrightarrow{f_1}J_{\CC}(M,X)\to0.\]
\end{itemize}
\end{lemma}
The condition (b) is a property of relative $n$-Auslander algebras.
The sequence in the condition (c) is given by an $n$-almost split sequences if $X$ is non-projective.

\subsection{Example: $n$-cluster tilting for $(n-1)$-Auslander algebras}\label{$n$-cluster tilting for higher Auslander algebras}

In this subsection we give a series of examples of finite-dimensional algebras with $n$-cluster tilting objects from \cite{I-inductive}.
Let us start with general results on $n$-cluster tilting subcategories of $\mod\Lambda$.
For a finite-dimensional algebra $\Lambda$, we define full subcategories of $\mod\Lambda$ by
\[\MM=\MM_n(\Lambda):=\add\{\tau_n^i(D\Lambda)\ |\ i\ge0\}\ \mbox{ and }\ \MM'=\MM'_n(\Lambda):=\add\{\tau_n^{-i}(\Lambda)\ |\ i\ge0\}.\]
Then the following result is an immediate consequence of Theorem \ref{n-AR translation}.
\begin{proposition}\label{tau_n closure}
\begin{itemize}
\item[(a)] Any $n$-cluster tilting subcategory $\CC$ of $\mod\Lambda$ contains $\MM$ and $\MM'$.
In particular, if $\mod\Lambda$ contains an $n$-cluster tilting subcategory, then $X\oplus Y$ is $n$-rigid for any $X\in\MM$ and $Y\in\MM'$.

\item[(b)] If $\mod\Lambda$ contains an $n$-cluster tilting object, then we have a bijection between indecomposable projective $\Lambda$-modules $I$
and indecomposable injective $\Lambda$-modules $I$.
It is given by $I\mapsto\tau_n^{\ell_I}I$, where $\ell_I$ is a maximal natural number $\ell$ satisfying $\tau_n^\ell I\neq0$ and $\tau_n^{\ell+1}I=0$.
\end{itemize}
\end{proposition}
In this subsection we construct a family of finite-dimensional algebras $\Lambda$ such that $\MM$ itself forms an $n$-cluster tilting subcategory of $\mod\Lambda$
(or $T^\perp$ for some tilting $\Lambda$-module $T$).
Let us start with the following typical examples.

\begin{example} \label{example-B}
Let $\Lambda_1$ and $\Lambda'_1$ be the path algebras of the quivers
{\small\[\bullet\longrightarrow\bullet\longrightarrow\bullet\ \ \ \mbox{ and }\ \ \
\bullet\longrightarrow\bullet\longleftarrow\bullet\ ,\ \mbox{ respectively,}\]}%
over a field $k$.
Then we have $\MM_1(\Lambda_1)=\mod\Lambda_1$ and $\MM_1(\Lambda'_1)=\mod\Lambda'_1$.
We denote by $\Lambda_2$ and $\Lambda'_2$ the Auslander algebras of $\Lambda_1$ and $\Lambda'_1$ respectively, so they are given by the translation quivers
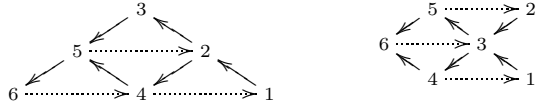
\[\xymatrix@C=0.5cm@R0.2cm{
&&{\scriptstyle 3}\ar[dl]\\
&{\scriptstyle 5}\ar[dl]\ar@{..>}[rr]&&{\scriptstyle 2}\ar[dl]\ar[ul]\\
{\scriptstyle 6}\ar@{..>}[rr]&&{\scriptstyle 4}\ar[ul]\ar@{..>}[rr]&&{\scriptstyle 1}\ar[ul]}\ \ \ \ \ \ \ \ \ \ \xymatrix@C=0.3cm@R0.1cm{
&{\scriptstyle 5}\ar[dl]\ar@{..>}[rr]&&{\scriptstyle 2}\ar[dl]\\
{\scriptstyle 6}\ar@{..>}[rr]&&{\scriptstyle 3}\ar[dl]\ar[ul]\\
&{\scriptstyle 4}\ar[ul]\ar@{..>}[rr]&&{\scriptstyle 1}\ar[ul]}\]
with the mesh relations.
We have
\begin{eqnarray*}
&D\Lambda_2=
\left(\begin{smallmatrix}
1
\end{smallmatrix}\oplus
\begin{smallmatrix}
1&\\
&2
\end{smallmatrix}\oplus
\begin{smallmatrix}
1&&\\
&2&\\
&&3
\end{smallmatrix}\oplus
\begin{smallmatrix}
&2\\
4&
\end{smallmatrix}\oplus
\begin{smallmatrix}
&2&\\
4&&3\\
&5&
\end{smallmatrix}\oplus
\begin{smallmatrix}
&&3\\
&5&\\
6&&
\end{smallmatrix}\right),&\\
&\tau_2D\Lambda_2=
\left(\begin{smallmatrix}
4
\end{smallmatrix}\oplus
\begin{smallmatrix}
4&\\
&5
\end{smallmatrix}\oplus
\begin{smallmatrix}
&5\\
6&
\end{smallmatrix}\right),\
\tau_2^2D\Lambda_2=
\left(\begin{smallmatrix}
6
\end{smallmatrix}\right),
\tau_2^3D\Lambda_2=0,&\\
&D\Lambda_2'=
\left(\begin{smallmatrix}
1
\end{smallmatrix}\oplus
\begin{smallmatrix}
2
\end{smallmatrix}\oplus
\begin{smallmatrix}
1&&2\\
&3&
\end{smallmatrix}\oplus
\begin{smallmatrix}
&&2\\
&3&\\
4&&
\end{smallmatrix}\oplus
\begin{smallmatrix}
1&&\\
&3&\\
&&5
\end{smallmatrix}\oplus
\begin{smallmatrix}
&3&\\
4&&5\\
&6&
\end{smallmatrix}\right),&\\
&\tau_2D\Lambda_2'=
\left(\begin{smallmatrix}
4
\end{smallmatrix}\oplus
\begin{smallmatrix}
5
\end{smallmatrix}\oplus
\begin{smallmatrix}
4&&5\\
&6&
\end{smallmatrix}\right),\
\tau_2^2D\Lambda_2'=0.&
\end{eqnarray*}
The quivers of $\MM_2(\Lambda_2)$ and $\MM_2(\Lambda_2')$ are
{\tiny
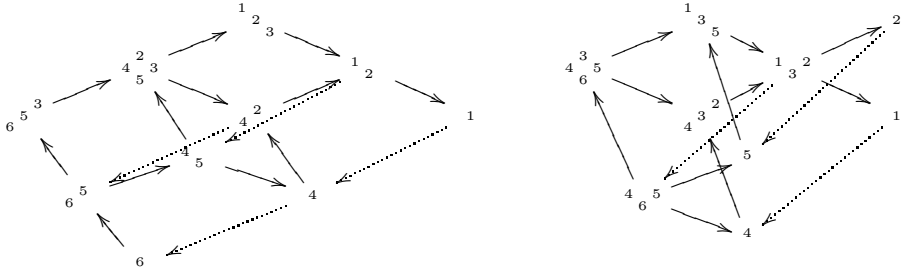
\[\xymatrix@C=0.0cm@R0.0cm{
&&&&{\begin{smallmatrix}
1&&\\
&2&\\
&&3
\end{smallmatrix}}\ar[drr]\\
&&{\begin{smallmatrix}
&2&\\
4&&3\\
&5&
\end{smallmatrix}}\ar[drr]\ar[urr]
&&&&{\begin{smallmatrix}
1&&\ \\
&2&
\end{smallmatrix}}\ar[drrrr]\ar@{..>}[ddlll]\\
{\begin{smallmatrix}
&&3\\
&5&\\
6&&
\end{smallmatrix}}\ar[urr]
&&&&{\begin{smallmatrix}
&2&\ \\
4&&
\end{smallmatrix}}\ar[urr]\ar@{..>}[ddlll]
&&&&&&{\begin{smallmatrix}
\ &1&\
\end{smallmatrix}}\ar@{..>}[ddlllll]\\
&&&{\begin{smallmatrix}
4&&\ \\
&5&
\end{smallmatrix}}\ar[uul]\ar[drr]\\
&{\begin{smallmatrix}
\ \\
&5&\ \\
6&&
\end{smallmatrix}}\ar[uul]\ar[urr]
&&&&{\begin{smallmatrix}
\ &4&\
\end{smallmatrix}}\ar[uul]\ar@{..>}[dddlll]\\
\ \\
\ \\
&&{\begin{smallmatrix}
\ \\
\ &6&\ \\
\
\end{smallmatrix}}\ar[uuul]}
\ \ \ \ \ \ \ \ \ \xymatrix@C=0.0cm@R0.0cm{
&&{\begin{smallmatrix}
1&&\\
&3&\\
&&5
\end{smallmatrix}}\ar[drr]
&&&&&&{\begin{smallmatrix}
2
\end{smallmatrix}}\ar@{..>}[dddlllll]\\
{\begin{smallmatrix}
&3&\\
4&&5\\
&6&
\end{smallmatrix}}\ar[drr]\ar[urr]
&&&&{\begin{smallmatrix}
1&&2\\
&3&
\end{smallmatrix}}\ar[drrrr]\ar[urrrr]\ar@{..>}[dddlll]\\
&&{\begin{smallmatrix}
&&2\\
&3&\\
4&&
\end{smallmatrix}}\ar[urr]
&&&&&&{\begin{smallmatrix}
\ &1&\
\end{smallmatrix}}\ar@{..>}[dddlllll]\\
&&&{\begin{smallmatrix}
\ \\
5\\
\
\end{smallmatrix}}\ar[uuul]\\
&{\begin{smallmatrix}
\ \\
4&&5\\
&6&
\end{smallmatrix}}\ar[drr]\ar[urr]\ar[uuul]\\
&&&{\begin{smallmatrix}
\ \\
4\\
\
\end{smallmatrix}}\ar[uuul]
}\]}%
where the dotted arrows indicate $\tau_2$.
Using Lemma \ref{criterion for n-cluster tilting}, one can verify that $\MM_2(\Lambda_2)$ is a 2-cluster tilting subcategory of $\mod\Lambda_2$, and that
$\MM_2(\Lambda_2')$ is a $2$-cluster tilting subcategory of $T_2^\perp$ for the tilting $\Lambda_2'$-module
\[T_2=\left(\begin{smallmatrix}
&&2\\
&3&\\
4&&
\end{smallmatrix}\oplus
\begin{smallmatrix}
1&&\\
&3&\\
&&5
\end{smallmatrix}\oplus
\begin{smallmatrix}
&3&\\
4&&5\\
&6&
\end{smallmatrix}\oplus
\begin{smallmatrix}
4
\end{smallmatrix}\oplus
\begin{smallmatrix}
5
\end{smallmatrix}\oplus
\begin{smallmatrix}
4&&5\\
&6&
\end{smallmatrix}\right)\]
with $\pd T_2=1$. Now let $\Lambda_3$ and $\Lambda_3'$ be the endomorphism algebras of the above $2$-cluster tilting objects.
Again one can verify that $\MM_3(\Lambda_3)$ is a $3$-cluster tilting subcategory of $\mod\Lambda_3$,
and that $\MM_3(\Lambda_3')$ is a $3$-cluster tilting subcategory of $T_3^\perp$ for some tilting $\Lambda_3'$-module $T_3$ with $\pd T_3=2$.

Moreover our inductive construction of $n$-cluster tilting objects continues infinitely!
\end{example}

Let us formalize the above example.
\begin{definition}\label{complete}
Let $\Lambda$ be a finite-dimensional algebra and $n$ a positive integer.
Define subcategories of $\MM=\MM_n(\Lambda)$ by
\begin{eqnarray*}
\PP(\MM)&:=&\{X\in\MM\ |\ \pd X<n\},\\
\MM_P&:=&\{X\in\MM\ |\ X\ \mbox{ has no non-zero summand in }\ \PP(\MM)\}.
\end{eqnarray*}
We call $\Lambda$ \emph{$n$-complete} if the following conditions are satisfied:
\begin{itemize}
\item[(a)] There exists a tilting $\Lambda$-module $T$ satisfying $\pd T<n$ and $\PP(\MM)=\add T$;
\item[(b)] there exists an $n$-cluster tilting object $M$ of $T^{\perp}$ satisfying $\MM=\add M$;
\item[(c)] $\gl\Lambda\le n$;
\item[(d)] $\Ext^i_\Lambda(\MM_P,\Lambda)=0$ for any $0<i<n$.
\end{itemize}
In this case, we call $\End_\Lambda(M)$ the \emph{cone} of $\Lambda$.
\end{definition}
It is easy to check that a finite-dimensional algebra is $1$-complete if and only if it is representation-finite and hereditary.
Our inductive construction is given by the following result.
\begin{theorem}
The cone of an $n$-complete algebra is $(n+1)$-complete.
\end{theorem}
A special case is the following result, which explains the above examples.
\begin{corollary}\label{inductive construction}
Let $Q$ be a Dynkin quiver and $\Lambda_1=kQ$ its path algebra over a field $k$. For any $n\ge1$, there exists an $n$-complete algebra $\Lambda_n$ with the cone $\Lambda_{n+1}$.
\end{corollary}
Thus, for each Dynkin quiver $Q$, we have an infinite series of $(n-1)$-relative $n$-Auslander algebras $\Lambda_n$ starting from the path algebra $\Lambda_1=kQ$.
If $Q$ is of Dynkin type $A_2$, then we have Example \ref{examples-A} (a).
The following case is especially interesting.
\begin{corollary}
Let $\Lambda_1$ be the $m\times m$ triangular matrix algebra over a field $k$ for some $m\ge1$.
For any $n\ge1$, there exist an algebra $\Lambda_n$ and an $n$-cluster tilting object $M_n$ in $\mod\Lambda_n$ such that $\Lambda_{n+1}=\End_{\Lambda_n}(M_n)$.
\end{corollary}
A lot of $n$-complete algebras will be constructed in a joint work \cite{IO} with Oppermann by using `$n$-APR tilting modules'.

\medskip
In the rest of this subsection we describe the quiver of $\Lambda_n$ in Corollary \ref{inductive construction}.
Let $I_x$ be the indecomposable injective $\Lambda_1$-module corresponding to the vertex $x\in Q_0$ and
\[\ell_x:=\sup\{\ell\ge0\ |\ \tau^\ell I_x\neq0\}.\]
Since $Q$ is a Dynkin quiver, we have $\ell_x<\infty$ for any $x\in Q_0$.
We put
\[\Delta_x:=\{(\ell_1,\cdots,\ell_n)\in\Z^n\ |\ \ell_1,\cdots,\ell_n\ge0,\ \ell_1+\cdots+\ell_n\le \ell_x\}.\]
For $1\le i\le n$, we put
\begin{eqnarray*}
\mbox{\boldmath $e$}_i&:=&(\stackrel{1}{0},\cdots,\stackrel{i-1}{0},\stackrel{i}{1},\stackrel{i+1}{0},\cdots,\stackrel{n}{0})\in\Z^n,\\
\mbox{\boldmath $v$}_i&:=&\left\{\begin{array}{cc}
-\mbox{\boldmath $e$}_i&i=1,\\
\mbox{\boldmath $e$}_{i-1}-\mbox{\boldmath $e$}_i&1<i\le n.
\end{array}\right.
\end{eqnarray*}
\begin{definition}
For a Dynkin quiver $Q$ and $n\ge1$, we define a quiver $Q^{(n)}=(Q_0^{(n)},Q_1^{(n)})$ as follows:
We put
\[Q_0^{(n)}:=\{(x,\mbox{\boldmath $\ell$})\ |\ x\in Q_0,\ \mbox{\boldmath $\ell$}\in\Delta_x\}.\]
There are the following $(n+1)$ kinds of arrows in $Q^{(n)}_1$:
\begin{itemize}
\item For any arrow $a:w\to x$ in $Q$, we have an arrow $(a^*,\mbox{\boldmath $\ell$}):(x,\mbox{\boldmath $\ell$})\to(w,\mbox{\boldmath $\ell$})$,
\item for any arrow $b:x\to y$ in $Q$, we have an arrow $(b,\mbox{\boldmath $\ell$}):(x,\mbox{\boldmath $\ell$})\to(y,\mbox{\boldmath $\ell$}+\mbox{\boldmath $v$}_1)$,
\item for any $1<i\le n$, we have an arrow $(x,\mbox{\boldmath $\ell$})_i:(x,\mbox{\boldmath $\ell$})\to(x,\mbox{\boldmath $\ell$}+\mbox{\boldmath $v$}_i)$,
\end{itemize}
if both the tail and the head belong to $Q_0^{(n)}$.

It is helpful to regard $x$ as a $0$-th entry of $(x,\mbox{\boldmath $\ell$})$, above $w$ as `$x-1$', and above $y$ as `$x+1$'.
\end{definition}
\begin{theorem}\label{n-AR quiver}
Under the circumstances in Corollary \ref{inductive construction}, the quiver of $\Lambda_n^{\rm op}$ is given by $Q^{(n)}$.
\end{theorem}
We end this subsection with the following example.
For simplicity, we denote by
\[x\ell_1\cdots \ell_n\ \ \ (\mbox{respectively, }\ \ x\ell_1\cdots\dot{\ell}_i\cdots\ell_n,\ \ \ a^*\ell_1\cdots \ell_n,\ \ \ b\ell_1\cdots \ell_n)\]
the vertex $(x,\mbox{\boldmath $\ell$})\in Q_0^{(n)}$ (respectively, the arrow $(x,\mbox{\boldmath $\ell$})_i$, $(a^*,\mbox{\boldmath $\ell$})$,
$(b,\mbox{\boldmath $\ell$})\in Q_1^{(n)}$) for $\mbox{\boldmath $\ell$}=(\ell_1,\cdots,\ell_n)$.

\begin{example} \label{example-C}
Let $Q$ be the quiver
${\scriptstyle\bf 1}\stackrel{a}{\longrightarrow}{\scriptstyle\bf 2}\stackrel{b}{\longrightarrow}{\scriptstyle\bf 3}\stackrel{c}{\longrightarrow}{\scriptstyle\bf 4}$ of Dynkin type $A_4$.
Then the quiver of $\Lambda_2^{\rm op}$ is the following
{\tiny
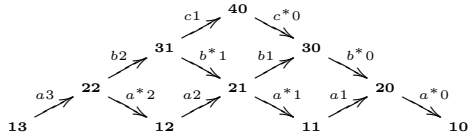
\[\xymatrix@C=0.5cm@R0.2cm{
&&&{\scriptstyle\bf 40}\ar[dr]^{c^*0}\\
&&{\scriptstyle\bf 31}\ar[ur]^{c1}\ar[dr]^{b^*1}&&{\scriptstyle\bf 30}\ar[dr]^{b^*0}\\
&{\scriptstyle\bf 22}\ar[ur]^{b2}\ar[dr]^{a^*2}&&{\scriptstyle\bf 21}\ar[ur]^{b1}\ar[dr]^{a^*1}&&{\scriptstyle\bf 20}\ar[dr]^{a^*0}\\
{\scriptstyle\bf 13}\ar[ur]^{a3}&&{\scriptstyle\bf 12}\ar[ur]^{a2}&&{\scriptstyle\bf 11}\ar[ur]^{a1}&&{\scriptstyle\bf 10}}\]}%
The quiver of $\Lambda_3^{\rm op}$ is the following
{\tiny
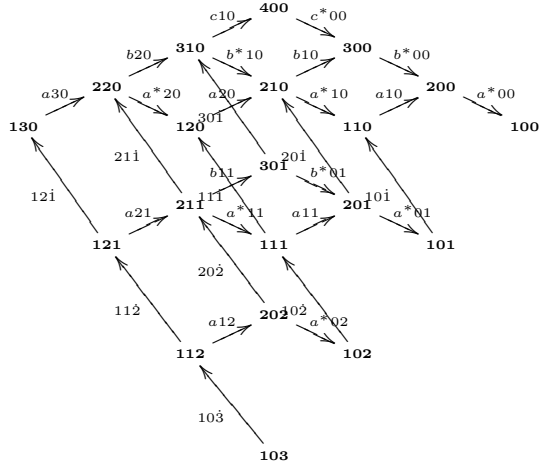
\[\xymatrix@C=0.5cm@R0.2cm{
&&&{\scriptstyle\bf 400}\ar[dr]^{c^*00}\\
&&{\scriptstyle\bf 310}\ar[ur]^{c10}\ar[dr]^{b^*10}&&{\scriptstyle\bf 300}\ar[dr]^{b^*00}\\
&{\scriptstyle\bf 220}\ar[ur]^{b20}\ar[dr]^{a^*20}&&{\scriptstyle\bf 210}\ar[ur]^{b10}\ar[dr]^{a^*10}&&{\scriptstyle\bf 200}\ar[dr]^{a^*00}\\
{\scriptstyle\bf 130}\ar[ur]^{a30}&&{\scriptstyle\bf 120}\ar[ur]^{a20}&&{\scriptstyle\bf 110}\ar[ur]^{a10}&&{\scriptstyle\bf 100}\\
 &&&{\scriptstyle\bf 301}\ar[dr]^{b^*01}\ar[uuul]^{30\dot{1}}\\
 &&{\scriptstyle\bf 211}\ar[ur]^{b11}\ar[dr]^{a^*11}\ar[uuul]^{21\dot{1}}&&{\scriptstyle\bf 201}\ar[dr]^{a^*01}\ar[uuul]^{20\dot{1}}\\
 &{\scriptstyle\bf 121}\ar[ur]^{a21}\ar[uuul]^{12\dot{1}}&&{\scriptstyle\bf 111}\ar[ur]^{a11}\ar[uuul]^{11\dot{1}}&&{\scriptstyle\bf 101}\ar[uuul]^{10\dot{1}}\\
\ \\
&&&{\scriptstyle\bf 202}\ar[dr]^{a^*02}\ar[uuul]^{20\dot{2}}\\
&&{\scriptstyle\bf 112}\ar[ur]^{a12}\ar[uuul]^{11\dot{2}}&&{\scriptstyle\bf 102}\ar[uuul]^{10\dot{2}}\\
\ \\
\ \\
&&&{\scriptstyle\bf 103}\ar[uuul]^{10\dot{3}}
}\]}%
The quiver of $\Lambda_4^{\rm op}$ is the following
{\tiny
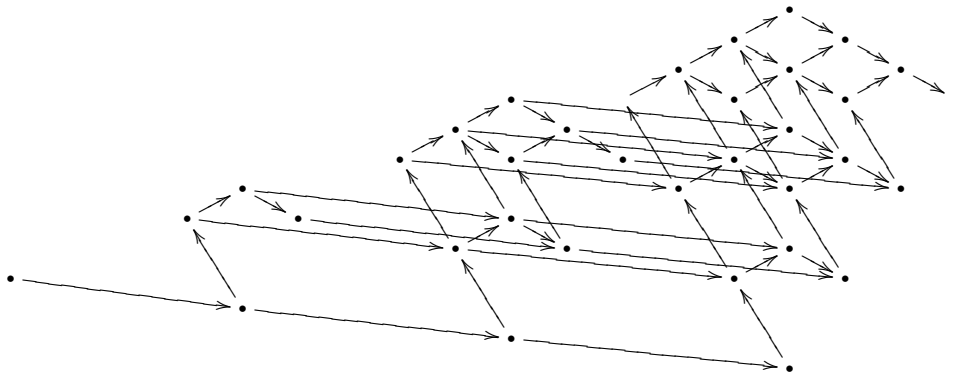
\[\xymatrix@C=0.4cm@R0.1cm{
&&&&&&&&&&&&&\bullet\ar[dr]&&&&\\
&&&&&&&&&&&&\bullet\ar[ur]\ar[dr]&&\bullet\ar[dr]&&&\\
&&&&&&&&&&&\bullet\ar[ur]\ar[dr]&&\bullet\ar[ur]\ar[dr]&&\bullet\ar[dr]&&\\
&&&&&&&&\bullet\ar[dr]\ar[drrrrr]&&
\ar[ur]&&\bullet\ar[ur]&&\bullet\ar[ur]&&\bullet\\
&&&&&&&\bullet\ar[ur]\ar[dr]\ar[drrrrr]&&\bullet\ar[dr]\ar[drrrrr]&&
&&\bullet\ar[dr]\ar[uuul]\\
&&&&&&\bullet\ar[ur]\ar[drrrrr]&&\bullet\ar[ur]\ar[drrrrr]&&\bullet\ar[drrrrr]
&&\bullet\ar[ur]\ar[dr]\ar[uuul]&&\bullet\ar[dr]\ar[uuul]&&\\
&&&\bullet\ar[dr]\ar[drrrrr]&&&&&&&
&\bullet\ar[ur]\ar[uuul]&&\bullet\ar[ur]\ar[uuul]&&\bullet\ar[uuul]&&\\
&&\bullet\ar[ur]\ar[drrrrr]&&\bullet\ar[drrrrr]&&
&&\bullet\ar[dr]\ar[uuul]\ar[drrrrr]\\
&&&&&&&\bullet\ar[ur]\ar[uuul]\ar[drrrrr]&&\bullet\ar[uuul]\ar[drrrrr]&
&&&\bullet\ar[dr]\ar[uuul]\\
\bullet\ar[drrr]&\ \ \ \ \ \ \ \ \ \ \ \ &&&&&&&&&
&&\bullet\ar[ur]\ar[uuul]&&\bullet\ar[uuul]\\
&&&\bullet\ar[uuul]\ar[drrrrr]\\
&&&&&&&&\bullet\ar[uuul]\ar[drrrrr]\\
&&&&&&&&&&&&&\bullet\ar[uuul]
}\]}%
In general, the shape of the quiver of $\Lambda_n^{\rm op}$ looks like an $n$-simplex.
\end{example}

\subsection{Questions}

In this subsection we shall pose a few questions on $n$-cluster tilting subcategories.
Throughout let $\Lambda$ be a finite-dimensional algebra.
By definition $\Lambda$ contains a $1$-cluster tilting object if and only if $\Lambda$ is representation-finite.
For $n\ge2$, the following question is not understood well.
\begin{itemize}
\item When does $\mod\Lambda$ contain an $n$-cluster tilting object (respectively, subcategory)?
\end{itemize}
A necessary conditions is given by Proposition \ref{tau_n closure}.
Let us give another necessary condition.
The \emph{complexity} of a $\Lambda$-module $M$ with a minimal projective resolution $\cdots\to P_1\to P_0\to M\to0$ is defined by
\[\inf\{\ell\ge0\ |\ \mbox{there exists $c>0$ such that }\dim_kP_i\le ci^{\ell-1}\ \mbox{ for any }i\}.\]

Erdmann and Holm \cite{EH} proved the following result.
\begin{theorem}
Let $\Lambda$ be a selfinjective algebra and $n\ge1$. If there exists an $n$-cluster tilting object in $\mod\Lambda$, then the complexity of any $\Lambda$-module is at most one.
\end{theorem}
The following question is also interesting.
\begin{itemize}
\item Does any $n$-cluster tilting subcategory $\CC$ of $\mod\Lambda$ with $n\ge 2$ have an additive generator?
\end{itemize}
This is related to the following question due to Auslander and Smal\o \ \cite{ASm2}:
\begin{itemize}
\item Let $\XX$ be a functorially finite extension-closed subcategory of $\mod\Lambda$.
Are there only finitely many isoclasses of indecomposable projective (respectively, injective) objects in $\XX$?
\end{itemize}
If this is true, then the previous question is also true since any $n$-cluster tilting subcategory satisfies the conditions and all objects are projective.

It seems that answer to the following simple question is also unknown.
\begin{itemize}
\item Is there a finite-dimensional algebra $\Lambda$ with an infinite set $\{ X_i\}_{i\in I}$ of isoclasses of indecomposable $\Lambda$-modules satisfying
$\Ext^1_\Lambda(X_i,X_j)=0$ for any $i,j\in I$?
\end{itemize}
It is shown in \cite{I-aorder} that if $\repdim\Lambda\le 3$, then there is no such an infinite set.
The following result \cite{I-inductive} is interesting also from the viewpoint of the above question.
\begin{proposition}
For any finite-dimensional algebra $\Lambda$, the subcategories $\MM_2(\Lambda)$ and $\MM'_2(\Lambda)$ are 2-rigid.
\end{proposition}
In Theorem \ref{derived equivalence} we shall show that the numbers of indecomposable direct summands of
all basic $2$-cluster tilting objects in $\mod\Lambda$ are equal.
Thus the following question is natural.
\begin{itemize}
\item Are the numbers of indecomposable direct summands of all basic $n$-cluster tilting objects in $\mod\Lambda$ equal for a fixed $n$?
\end{itemize}

\section{Orders and Cohen-Macaulay modules}\label{Orders and Cohen-Macaulay modules}

In this section
we consider a large class of algebras, containing the finite-dimensional algebras, which are called \emph{orders} over complete local rings with Krull dimension $d$.
We briefly recall the Auslander-Reiten theory for the category of (maximal) \emph{Cohen-Macaulay} modules over orders.
The case $d=0$ is what we recalled in Section \ref{From Auslander-Reiten theory}.
The case $d=1$ was developped by many authors including Drozd, Kirichenko, Roiter, Reiner, Roggenkamp, Rump, Simson, Hijikata, Nishida (see \cite{DK,DKR,CRe1,CRe2,RH,Rog,Ru1,Ru2,Si,HN1,HN2}).
The case $d\ge 2$ is studied mainly in the context of commutative ring theory (see the references in \cite{Y}) or non-commutative algebraic geometry (see \cite{Ar1,Ar2,ASc,GN1,GN2,RV1,V1,V2}).
An especially nice phenomenon appears for the case $d=2$. In this case, the category $\CM(\Lambda)$ has fundamental sequences
connecting indecomposable projective modules and indecomposable injective modules in $\CM(\Lambda)$ which behave like almost split sequences.

In Subsection \ref{$n$-almost split sequences and $n$-Auslander algebras} we introduce $n$-Auslander-Reiten theory on $n$-cluster tilting subcategories $\CC$ of $\CM(\Lambda)$
by constructing the $n$-Auslander-Reiten translation, the $n$-Auslander-Reiten duality and $n$-almost split sequences.
For the case $n=d-1$ we shall show the existence of $(d-1)$-fundamental sequences connecting indecomposable projective objects in $\CC$ and indecomposable injective objects in $\CC$.
This implies that the $(d-1)$-Auslander algebras $\End_\Lambda(M)$ of $(d-1)$-cluster tilting objects $M\in\CM(\Lambda)$ are especially nice.
In fact they are $R$-orders with global dimension $d$.
If in addition $\Lambda$ is a symmetric $R$-order, then $\End_\Lambda(M)$ is a $d$-Calabi-Yau algebra.
In Subsections \ref{Cohen-Macaulay modules and triangulated categories} and \ref{2-cluster tilting for 2-Calabi-Yau categories}
we recall several results on Calabi-Yau algebras and Calabi-Yau triangulated categories given by orders.

Our basic references in this section are \cite{BH,Ma} for commutative rings,
\cite{A-order,A-isolated,Y} for the representation theory of orders,
and \cite{CRe1,GN2} for module-finite algebras.

\subsection{From Auslander-Reiten theory}

In this section we denote by $R$ a noetherian complete local ring of Krull dimension $d$.
For $X\in\mod R$, let us introduce the invariants
\begin{eqnarray*}
\dim X&:=&\dim(R/\ann X),\\
\depth X&:=&\inf\{i\ge0\ |\ \Ext^i_R(R/J_R,X)\neq0\},
\end{eqnarray*}
called the \emph{dimension} and the \emph{depth}.
We always have an inequality $\depth X\le \dim X$, and an important class of $R$-modules is the category
\[\CM_i(R):=\{X\in\mod\Lambda\ |\ X=0\ \mbox{ or }\ \dim X=i=\depth X\}\]
of \emph{Cohen-Macaulay modules of Krull dimension $i$}.
We are interested in the category
\[\CM(R):=\CM_d(R)\]
of \emph{(maximal) Cohen-Macaulay} $R$-modules.
We call $R$ a \emph{Cohen-Macaulay ring} if $R\in\CM(R)$.

\medskip
In the rest of this section we assume that $R$ is a \emph{Gorenstein ring},
that is the injective dimension $\id R$ of the $R$-module $R$ is finite, or equivalently $\id R=d$.
Any Gorenstein ring is Cohen-Macaulay.
An important class of Gorenstein rings is given by \emph{regular rings} $R$, for which $\gl R<\infty$, or equivalently $\gl R=d$.
A typical example of regular rings is the formal power series ring $k[[t_1,\cdots,t_d]]$ over a field $k$.
Conversely, any complete regular local ring containing a field must be a formal power series ring over a field by Cohen's structure theorem.
Since we have the Auslander-Buchsbaum formula
\begin{eqnarray}\label{Auslander-Buchsbaum}
\depth X+\pd X=\depth R\ \mbox{ if $\pd X<\infty$},
\end{eqnarray}
in general, we have $\CM(R)=\add R$ if $R$ is regular.

We have a nice duality theory for Gorenstein rings $R$.
In particular we have a duality
\[D_i:=\Ext^{d-i}_R(-,R):\CM_i(R)\xrightarrow{\sim}\CM_i(R).\]
The category $\CM_0(R)$ coincides with the category $\fl R$ of finite length $R$-modules,
and the duality
\[D:=D_0=\Ext^d_R(-,R):\fl R\xrightarrow{\sim}\fl R\]
is called the \emph{Matlis duality}.

\medskip
Let us introduce the algebras which we are interested in.
Let $\Lambda$ be a \emph{module-finite} $R$-algebra, namely an $R$-algebra which is finitely generated as an $R$-module (and not necessarily commutative).
Then $\Lambda$ is clearly noetherian, and the category $\mod\Lambda$ is Krull-Schmidt in the sense of
Introduction
since we assumed that the base ring $R$ is complete \cite{CRe1}.
We put
\[\CM(\Lambda):=\{X\in\mod\Lambda\ |\ X\in\CM(R) \mbox{ as an $R$-module}\}.\]
The category $\CM(\Lambda)$ is independent of the choice of the central subring $R$ of $\Lambda$ since so are $\dim X$ and $\depth X$.
\begin{definition}
We call $\Lambda$ an \emph{$R$-order} if $\Lambda\in\CM(\Lambda)$.
\end{definition}
In the rest of this subsection we assume that $\Lambda$ is an $R$-order.
The duality $D_d:\CM(R)\to\CM(R)$ induces a duality
\[D_d=\Hom_R(-,R):\CM(\Lambda)\stackrel{\sim}{\longleftrightarrow}\CM(\Lambda^{\rm op}).\]
Thus $\CM(\Lambda)$ forms an extension-closed subcategory of $\mod\Lambda$ with enough projectives $\add\Lambda$ and enough injectives $\add D_d\Lambda$.
We have the \emph{Nakayama functors}
\[\nu:=D_d(-)^*:\mod\Lambda\to\mod\Lambda\ \mbox{ and }\ \nu^-:=(-)^*D_d:\mod\Lambda\to\mod\Lambda,\]
which induce mutually quasi-inverse equivalences
\[\nu:\add\Lambda\xrightarrow{\sim}\add D_d\Lambda\ \mbox{ and }\ \nu^-:\add D_d\Lambda\xrightarrow{\sim}\add\Lambda.\]

\begin{examples} \label{examples-D}
(a) An $R$-order of Krull dimension $d=0$ is nothing but an artin $R$-algebra, and we have $\mod\Lambda=\CM(\Lambda)$ in this case.

(b) Let $S$ be a noetherian commutative complete local ring containing a field.
Then $S$ always contains a complete regular local subring $R$ such that $S$ is a module-finite $R$-algebra.
In this case, $S$ is an $R$-order if and only if $S$ is a Cohen-Macaulay ring.
Moreover, an $S$-module is Cohen-Macaulay as an $S$-module if and only if it is Cohen-Macaulay as an $R$-module.
\end{examples}

Since $R$ is Goresntein, we have
\begin{eqnarray*}
\dim X&=&d-\inf\{i\ge0\ |\ \Ext^i_R(X,R)\neq0\}\\
&=&d-\inf\{i\ge0\ |\ \Ext^i_\Lambda(X,D_d\Lambda)\neq0\},\\
\depth X&=&d-\sup\{i\ge0\ |\ \Ext^i_R(X,R)\neq0\}\\
&=&d-\sup\{i\ge0\ |\ \Ext^i_\Lambda(X,D_d\Lambda)\neq0\}.
\end{eqnarray*}
In particular, we have
\[\CM(\Lambda)=\{X\in\mod\Lambda\ |\ \Ext^i_\Lambda(X,D_d\Lambda)=0 \mbox{ for } i>0\}.\]
Thus the study of Cohen-Macaulay modules is a special case of cotilting theory with respect to $D_d\Lambda$ \cite{AR-contravariant,ABu}.

For an $R$-order $\Lambda$, we always have an inequality $\gl\Lambda\ge d$.
\begin{definition}\label{definition isolated singularity}
An $R$-order $\Lambda$ is called \emph{non-singular} \cite{A-isolated} if $\gl\Lambda=d$.

An $R$-order $\Lambda$ is called an \emph{isolated singularity} \cite{A-isolated} if
\[\gl(\Lambda\otimes_RR_{\mathfrak p})=\dim R_{\mathfrak p}\]
for any non-maximal prime ideal ${\mathfrak p}$ of $R$.
For example, non-singular $R$-orders are isolated singularities.
\end{definition}
In the rest of this subsection we assume that $\Lambda$ is an $R$-order which is an isolated singularity.
For example, an $R$-order of Krull dimension $d=0$ is always an isolated singularity, and
an $R$-order of Krull dimension $d=1$ is an isolated sigularity if and only if $\Lambda\otimes_RK$ is a semisimple $K$-algebra for the quotient field $K$ of $R$.
As for $\mod\Lambda$ for finite-dimensional algebras $\Lambda$, many results in Auslander-Reiten theory hold in $\CM(\Lambda)$
for an $R$-order $\Lambda$ which is an isolated singularity.
We put
\[\underline{\CM}(\Lambda):=\CM(\Lambda)/[\Lambda]\ \mbox{ and }\ \overline{\CM}(\Lambda):=\CM(\Lambda)/[D_d\Lambda],\]
where $[\Lambda]$ and $[D_d\Lambda]$ are ideals of $\CM(\Lambda)$ defined in
Introduction.
We notice that isolated singularities are characterized in terms of stable categories as follows \cite{A-isolated}.
\begin{itemize}
\item An $R$-order $\Lambda$ is non-singular if and only if $\CM(\Lambda)=\add\Lambda$ if and only if $\underline{\CM}(\Lambda)=0$.
\item An $R$-order $\Lambda$ is an isolated singularity if and only if the morphism set $\underline{\Hom}_\Lambda(X,Y)$ in $\underline{\CM}(\Lambda)$ belongs to $\fl R$ for any $X,Y\in\CM(\Lambda)$.
\end{itemize}
We denote by
\[\Omega:\underline{\CM}(\Lambda)\to\underline{\CM}(\Lambda)\ \mbox{ and }\ \Omega^-:\overline{\CM}(\Lambda)\to\overline{\CM}(\Lambda)\]
the syzygy and the cosyzygy functors.
A key observation is the following, which gives the Auslander-Bridger transpose duality for $\CM(\Lambda)$ \cite{A-order} (see \cite[1.1.1, 1.3.1]{I-aorder} for a short proof).
\begin{theorem}
We have a duality
\[\Omega^d\Tr:\underline{\CM}(\Lambda)\stackrel{\sim}{\longleftrightarrow}\underline{\CM}(\Lambda^{\rm op})\]
satisfying $(\Omega^d\Tr)^2\simeq 1_{\underline{\CM}(\Lambda)}$.
\end{theorem}
Since the duality $D_d:\CM(\Lambda)\leftrightarrow\CM(\Lambda^{\rm op})$ induces a duality $D_d:\underline{\CM}(\Lambda)\leftrightarrow\overline{\CM}(\Lambda^{\rm op})$,
we have mutually quasi-inverse equivalences
\begin{eqnarray*}
\tau:=D_d\Omega^d\Tr&:&\underline{\CM}(\Lambda)\xrightarrow{\Omega^d\Tr}\underline{\CM}(\Lambda^{\rm op})\xrightarrow{D_d}\overline{\CM}(\Lambda),\\
\tau^-:=\Omega^d\Tr D_d&:&\overline{\CM}(\Lambda)\xrightarrow{D_d}\underline{\CM}(\Lambda^{\rm op})\xrightarrow{\Omega^d\Tr}\underline{\CM}(\Lambda),
\end{eqnarray*}
called the \emph{Auslander-Reiten translations}.
As in the case of finite-dimensional algebras,
$\tau$ gives a bijection from isoclasses of indecomposable non-projective objects in $\CM(\Lambda)$ to
isoclasses of indecomposable non-injective objects in $\CM(\Lambda)$.
Moreover, we have the following results.
\begin{theorem}\label{AR theory for orders}
\begin{itemize}
\item[(a)] There exist the following functorial isomorphisms for any $X,Y\in\CM(\Lambda)$ (called the \emph{Auslander-Reited duality})
\[\underline{\Hom}_\Lambda(\tau^-Y,X)\simeq D\Ext^1_\Lambda(X,Y)\simeq\overline{\Hom}_\Lambda(Y,\tau X).\]
\item[(b)] For any indecomposable non-projective $X\in\CM(\Lambda)$ (respectively, non-injective $Z\in\CM(\Lambda)$), there exists an exact sequence (called an \emph{almost split sequence})
\[0\to Z\xrightarrow{g}Y\xrightarrow{f}X\to0\]
with terms in $\CM(\Lambda)$ which induces the following exact sequences on $\CM(\Lambda)$
\begin{eqnarray*}
&0\to(-,Z)\xrightarrow{g}(-,Y)\xrightarrow{f}J_{\CM(\Lambda)}(-,X)\to0,&\\
&0\to(X,-)\xrightarrow{f}(Y,-)\xrightarrow{g}J_{\CM(\Lambda)}(Z,-)\to0.&
\end{eqnarray*}
Moreover, we have $Z\simeq\tau X$ (respectively, $X\simeq\tau^-Z$).
\end{itemize}
\end{theorem}
The category $\CM(\Lambda)$ is especially nice for the case $d=2$.
\begin{theorem}\label{fundamental}
If $d=2$, then we have the following.
\begin{itemize}
\item[(a)] The Nakayama functor induces an equivalence $\nu:\CM(\Lambda)\to\CM(\Lambda)$ which makes the following diagram commutative
\[\begin{array}{ccc}
\CM(\Lambda)&\xrightarrow{\nu}&\CM(\Lambda)\\
\downarrow&&\downarrow\\
\underline{\CM}(\Lambda)&\xrightarrow{\tau}&\overline{\CM}(\Lambda).
\end{array}\]
\item[(b)] For any projective $P\in\CM(\Lambda)$, there exists an exact sequence (called a \emph{fundamental sequence})
\[0\to \nu P\xrightarrow{g}C\xrightarrow{f}P\to P/PJ_\Lambda\to0\]
which induces the following exact sequences on $\CM(\Lambda)$
\begin{eqnarray*}
&0\to(-,\nu P)\xrightarrow{g}(-,C)\xrightarrow{f}J_{\CM(\Lambda)}(-,P)\to0,&\\
&0\to(P,-)\xrightarrow{f}(C,-)\xrightarrow{g}J_{\CM(\Lambda)}(\nu P,-)\to0.&
\end{eqnarray*}
\end{itemize}
\end{theorem}
The existence of fundamental sequences plays a crucial role in the classification theory of representation-finite orders of Krull dimension $d=2$
due to Reiten and Van den Bergh \cite{RV1}.
In this case, the structure theory of $\CM(\Lambda)$ is much nicer than in other dimensions $d\neq2$ (including $d=0,1$)
since fundamental sequences play the role of almost split sequences for projective $\Lambda$-modules (see \cite{I-tau12-I,I-tau12-II,I-tau3}),
and the whole Auslander-Reiten quiver of $\CM(\Lambda)$ always can be regarded as a stable translation quiver.

\subsection{$n$-Auslander-Reiten theory}\label{$n$-almost split sequences and $n$-Auslander algebras}

In this subsection we consider $n$-cluster tilting subcategories of $\CM(\Lambda)$.
As in the case of finite-dimensional algebras, we consider the functors
\begin{eqnarray*}
\tau_n:=\tau\Omega^{n-1}&:&\underline{\CM}(\Lambda)\xrightarrow{\Omega^{n-1}}\underline{\CM}(\Lambda)\xrightarrow{\tau}\overline{\CM}(\Lambda),\\
\tau_n^-:=\tau^-\Omega^{-(n-1)}&:&\overline{\CM}(\Lambda)\xrightarrow{\Omega^{-(n-1)}}\overline{\CM}(\Lambda)\xrightarrow{\tau^-}\underline{\CM}(\Lambda).
\end{eqnarray*}
As in the case of finite-dimensional algebras, we have the following results.

\begin{theorem}
Let $\CC$ be an $n$-cluster tilting subcategory of $\CM(\Lambda)$.
\begin{itemize}
\item[(a)] $\tau_n$ and $\tau_n^-$ induce mutually quasi-inverse equivalences $\tau_n:\underline{\CC}\to\overline{\CC}$ and $\tau_n^-:\overline{\CC}\to\underline{\CC}$
(called the \emph{$n$-Auslander-Reiten translations}).
In particular, $\tau_n$ gives a bijection from isoclasses of indecomposable non-projective objects in $\CC$ to
isoclasses of indecomposable non-injective objects in $\CC$.
\item[(b)] There exist the following functorial isomorphisms for any $X,Y\in\CC$ (called the \emph{$n$-Auslander-Reiten duality})
\[\underline{\Hom}_\Lambda(\tau_n^-Y,X)\simeq D\Ext^n_\Lambda(X,Y)\simeq\overline{\Hom}_\Lambda(Y,\tau_nX).\]
\item[(c)] Any indecomposable non-projective $X\in\CC$ (respectively, non-injective $Y\in\CC$) has an exact sequence (called an \emph{$n$-almost split sequence})
\[0\to Y\xrightarrow{f_{n+1}}C_n\xrightarrow{f_n}\cdots\xrightarrow{f_2}C_1\xrightarrow{f_1}X\to0\]
with terms in $\CC$ such that $f_i\in J_{\CC}$ for any $i$ and the following sequences are exact on $\CC$
\begin{eqnarray*}
&0\to(-,Y)\xrightarrow{f_{n+1}}(-,C_n)\xrightarrow{f_n}\cdots\xrightarrow{f_2}(-,C_1)\xrightarrow{f_1}J_{\CC}(-,X)\to0,&\\
&0\to(X,-)\xrightarrow{f_1}(C_1,-)\xrightarrow{f_2}\cdots\xrightarrow{f_n}(C_n,-)\xrightarrow{f_{n+1}}J_{\CC}(Y,-)\to0.&
\end{eqnarray*}
Moreover, we have $Y\simeq\tau_nX$ (respectively, $X\simeq\tau_n^-Y$).
\end{itemize}
\end{theorem}
The case $n=d-1$ is especially nice and we have the following analogous result to Theorem \ref{fundamental}.
\begin{theorem}\label{(d-1)-fundamental sequence}
Let $\CC$ be a $(d-1)$-cluster tilting subcategory in $\CM(\Lambda)$.
\begin{itemize}
\item[(a)] The Nakayama functor induces an equivalence $\nu:\CC\to\CC$ which makes the following diagram commutative
\[\begin{array}{ccc}
\CC&\xrightarrow{\nu}&\CC\\
\downarrow&&\downarrow\\
\underline{\CC}&\xrightarrow{\tau_{d-1}}&\overline{\CC}
\end{array}.\]
\item[(b)] For any indecomposable projective $P\in\CC$, there exists an exact sequence (called a \emph{$(d-1)$-fundamental sequence})
\[0\to\nu P\xrightarrow{f_{n+1}}C_{n}\xrightarrow{f_n}\cdots\xrightarrow{f_2}C_1\stackrel{f_1}{\to}P\to P/PJ_\Lambda\to 0\]
with terms in $\CC$ such that $f_i\in J_{\CC}$ for any $i$ and the following sequences are exact on $\CC$
\begin{eqnarray*}
&0\to(-,\nu P)\xrightarrow{f_{n+1}}(-,C_n)\xrightarrow{f_n}\cdots\xrightarrow{f_2}(-,C_1)\xrightarrow{f_1}J_{\CC}(-,P)\to0&\\
&0\to(P,-)\xrightarrow{f_1}(C_1,-)\xrightarrow{f_2}\cdots\xrightarrow{f_n}(C_n,-)\xrightarrow{f_{n+1}}J_{\CC}(\nu P,-)\to0.&
\end{eqnarray*}
\end{itemize}
\end{theorem}
Let us give an example of $(d-1)$-cluster tilting object.

\begin{example}
Let $k$ be a field of characteristic zero, and $S=k[[x_1,\cdots,x_d]]$ the formal power series ring.
Let $G$ be a finite subgroup of $\GL_d(k)$.
Then $G$ naturally acts on the $k$-vector space $V:=k^d$, and the action extends to the completion $S$ of the symmetric algebra of $V$. Let
\[\Lambda:=S^G=\{f\in S\ |\ f^\sigma=f\ \mbox{ for any }\ \sigma\in G\}\]
be the invariant subring of $S$.
Assume that $\Lambda$ satisfies the following conditions:
\begin{itemize}
\item $\Lambda$ is an isolated singularity;
\item if $\sigma\in G$ is a pseudo-reflection (namely $\rank(\sigma-1)\le 1$), then $\sigma=1$.
\end{itemize}
\end{example}
The second condition implies that  $\End_{\Lambda}(S)$ is naturally
isomorphic to the skew group ring $S * G$ \cite{A-rational,Y}
(the assumption $d = 2$ there was not used).
Then we have the following result \cite{I-ar,I-aorder}.
\begin{theorem}\label{invariant subring}
\begin{itemize}
\item[(a)] $S$ is a $(d-1)$-cluster tilting object in $\CM(\Lambda)$.
\item[(b)] The Koszul complex
\[0\to S\otimes_k(\wedge^dV)\to S\otimes_k(\wedge^{d-1}V)\to\cdots\to S\otimes_kV\to S\to k\to0\]
is a direct sum of $(d-1)$-almost split sequences and $(d-1)$-fundamental sequences in the $(d-1)$-cluster tilting subcategory $\add S_\Lambda$ in $\CM(\Lambda)$.
\item[(c)] The quiver of the category $\add S_\Lambda$ coincides with the McKay quiver of $G$.
\end{itemize}
\end{theorem}
This generalizes famous results due to Herzog \cite{He} and Auslander \cite{A-rational} for Krull dimension $d=2$.
In this case, $S$ is a $1$-cluster tilting object in $\CM(\Lambda)$, so we have $\add S=\CM(\Lambda)$.
In particular, $\Lambda$ is representation-finite.

It is interesting to classify all $(d-1)$-cluster tilting objects in $\CM(\Lambda)$.
We have the following result \cite{IY,KR2,KV}.
\begin{theorem}
\begin{itemize}
\item[(a)] Let $G$ be a cyclic subgroup of $\GL_3(k)$ generated by $\sigma={\rm diag}(\omega,\omega,\omega)$ with the primitive third root of unity $\omega$.
Put
\[S_i:=\{f\in S\ |\ f^\sigma=\omega^if\}\]
for $i\in\Z/3\Z$.
Then all basic 2-cluster tilting objects in $\CM(\Lambda)$ are given by
$S_0\oplus\Omega^iS_1\oplus\Omega^iS_2$ and $S_0\oplus\Omega^{i+1}S_1\oplus\Omega^iS_2$, for $i\in\Z$.
\item[(b)] Let $G$ be a cyclic subgroup of $\GL_4(k)$ generated by $\sigma={\rm diag}(-1,-1,-1,-1)$.
Put
\[S_i:=\{f\in S\ |\ f^\sigma=(-1)^if\}\]
for $i\in\Z/2\Z$.
Then all basic 3-cluster tilting objects in $\CM(\Lambda)$ are given by $S_0\oplus\Omega^iS_1$, for $i\in\Z$.
\end{itemize}
\end{theorem}
In these cases, any $2$-rigid object in $\CM(\Lambda)$ is a direct summand of a direct sum of copies of some $(d-1)$-cluster tilting object.
This is not the case in general.

\medskip
In the rest of this subsection we study endomorphism algebras of $n$-cluster tilting objects in $\CM(\Lambda)$ \cite{ARo,I-aorder}.
We have $\Hom_\Lambda(X,Y)\in\CM(R)$ for any $X,Y\in\CM(\Lambda)$ for $d\le 2$, but this is not the case for $d\ge3$ (see \cite[2.5.1]{I-ar}).
Thus the general characterization of endomorphism algebras of $n$-cluster tilting objects in $\CM(\Lambda)$ becomes rather complicated.
Nevertheless the case $d\le n$ is quite nice.

We call an $R$-algebra $\Gamma$ an \emph{$n$-Auslander algebra} if
\begin{itemize}
\item $\Gamma$ is an $R$-order which is an isolated singularity;
\item $\Gamma$ satisfies $\gl\Gamma\le n+1$ and the two-sided $(d+1,n+1)$-condition in Definition \ref{Auslander-type condition}.
\end{itemize}
This is consistent with the definition of $n$-Auslander algebras in Subsection \ref{Auslander algebras}.
We have the following characterization of $n$-Auslander algebras, where we call an additive category \emph{$n$-cluster tilting} if it is equivalent to an $n$-cluster tilting subcategory of $\CM(\Lambda)$ for an $R$-order $\Lambda$ which is an isolated singularity.
\begin{theorem}\label{higher auslander correspondence 2}
Assume $d\le n$.
\begin{itemize}
\item[(a)] Let $M\in\CM(\Lambda)$ be an $n$-cluster tilting object. Then $\End_\Lambda(M)$ is an $n$-Auslander algebra.
\item[(b)] Any $n$-Auslander algebra is obtained in this way.
This gives a bijection between the sets of equivalence classes of $n$-cluster tilting categories with additive generators and Morita-equivalence classes of $n$-Auslander algebras.
\end{itemize}
\end{theorem}
Let us explain another converse of (a).
Later we shall use the following characterization \cite{I-aorder} of $n$-cluster tilting objects, which is analogous to Lemma \ref{criterion for n-cluster tilting}.
\begin{lemma}\label{criterion for n-cluster tilting 2}
Let $M\in\CM(\Lambda)$ be an $n$-rigid generator-cogenerator in $\CM(\Lambda)$ with $d\le n+2$.
Then the following conditions are equivalent.
\begin{itemize}
\item[(a)] $M$ is an $n$-cluster tilting object in $\CM(\Lambda)$.
\item[(b)] $\gl\End_\Lambda(M)\le n+1$.
\item[(c)] For any indecomposable object $X\in\add M$, there exists an exact sequence
\[0\to C_{n+1}\xrightarrow{f_{n+1}}C_n\xrightarrow{f_n}\cdots\xrightarrow{f_2}C_1\xrightarrow{f_1}X\]
with terms in $\add M$ such that the following sequence is exact on $\add M$
\[0\to(M,C_{n+1})\xrightarrow{f_{n+1}}(M,C_n)\xrightarrow{f_n}\cdots\xrightarrow{f_2}(M,C_1)\xrightarrow{f_1}J_{\CC}(M,X)\to0.\]
\end{itemize}
\end{lemma}
We have observed in Theorem \ref{(d-1)-fundamental sequence} that $(d-1)$-cluster tilting subcategories in $\CM(\Lambda)$ have $(d-1)$-fundamental sequences.
Consequently the endomorphism algebras of $(d-1)$-cluster tilting objects enjoys especially nice properties.
\begin{theorem}\label{Auslander algebra for n=d-1}
Let $M$ be a $(d-1)$-cluster tilting object in $\CM(\Lambda)$. Then $\End_\Lambda(M)$
is a non-singular $R$-order (Definition \ref{definition isolated singularity}) and satisfies the Gorenstein condition (Definition \ref{Gorenstein condition}).
\end{theorem}
Conversely, if $\Gamma$ is a non-singular $R$-order, then $\CM(\Gamma)=\add\Gamma$ holds and
$\Gamma$ is an $n$-cluster tilting object in $\CM(\Gamma)$ for arbitrary $n$.
In particular, Theorem \ref{higher auslander correspondence 2} holds also for $n=d-1$, since any $R$-order which is an isolated singularity satisfies the two-sided $(d,d)$-condition by \cite{GN1}.

\medskip
In general, there are a lot of $n$-cluster tilting subcategories in a fixed category $\CM(\Lambda)$,
and it is natural to study their relationship.
For the case $n=2$ (and arbitrary $d$), they are related by derived equivalence \cite{I-ar,L}.
\begin{theorem}\label{derived equivalence}
Let $M$ and $N$ be $2$-cluster tilting objects in $\CM(\Lambda)$ with $\Gamma:=\End_\Lambda(M)$ and $\Gamma':=\End_\Lambda(N)$.
Then $T:=\Hom_\Lambda(M,N)$ is a tilting $\Gamma$-module of projective dimension at most one with $\End_{\Gamma}(T)\simeq\Gamma'$.
In particular, $\Gamma$ and $\Gamma'$ are derived equivalent.
\end{theorem}
In particular, the numbers of indecomposable direct summands of all basic $2$-cluster tilting objects in $\CM(\Lambda)$ are equal.
Moreover, we have a relationship between 2-cluster tilting objects and tilting modules of projective dimension at most one.

Our Theorem \ref{derived equivalence} reminds us of a conjecture of Van den Bergh \cite{V1,V2},
who introduced a certain class of non-commutative algebras to study a conjecture of Bondal and Orlov \cite{BO} in birational geometry.
In the rest of this subsection we assume that $R$ is a normal domain.
Then reflexive $R$-modules behave nicely, for example they are closed under kernels and extensions.
\begin{definition}
Let $\Lambda$ be a module-finite $R$-algebra.
We say that $M\in\mod\Lambda$ gives a \emph{non-commutative crepant resolution} (or \emph{NCCR}) $\End_\Lambda(M)$ of $\Lambda$ if
\begin{itemize}
\item[(a)] $M$ is a reflexive $R$-module and $\End_\Lambda(M)$ is a non-singular $R$-order;
\item[(b)] $M_{\mathfrak p}$ is a generator of $\Lambda_{\mathfrak p}$ for any height one prime ideal ${\mathfrak p}$ of $R$.
\end{itemize}
\end{definition}
Notice that the condition (b) is not assumed in \cite{V1,V2} where the case $\Lambda=R$ is treated and (b) is automatically satisfied.
Van den Bergh conjectured that all (commutative and non-commutative) crepant resolutions of $\Lambda=R$ are derived equivalent,
and proved this for 3-dimensional terminal singularities.
Easy examples suggest that it is appropriate to impose the condition (b) to treat non-commutative algebras $\Lambda$.

Comparing the above condition (a) and Theorem \ref{Auslander algebra for n=d-1}, it is natural to hope that NCCR is related to $(d-1)$-cluster tilting objects.
In fact, we have the following result \cite{I-ar}.
\begin{theorem}
Let $\Lambda$ be an $R$-order which is an isolated singularity.
The following conditions are equvalent for $M\in\CM(\Lambda)$.
\begin{itemize}
\item[(a)] $M$ is a $(d-1)$-cluster tilting object in $\CM(\Lambda)$.
\item[(b)] $M$ is a generator-cogenerator in $\CM(\Lambda)$ and gives a NCCR of $\Lambda$.
\end{itemize}
\end{theorem}
Since modules giving NCCR are not necessarily Cohen-Macaulay, they are more general than $(d-1)$-cluster tilting objects in $\CM(\Lambda)$.
But the case $d=3$ is especially nice, and we have the following stronger result than Theorem \ref{derived equivalence} \cite{IR}.
\begin{theorem}\label{derived equivalence 2}
Let $d=3$ and $\Lambda$ be a module-finite $R$-algebra.
Assume that $M$ and $N$ give NCCR, $\Gamma:=\End_\Lambda(M)$ and $\Gamma'=\End_\Lambda(N)$ of $\Lambda$, respectively.
Then $T:=\Hom_\Lambda(M,N)$ is a tilting $\Gamma$-module of projective dimension at most one with $\End_{\Gamma}(T)\simeq\Gamma'$.
In particular, $\Gamma$ and $\Gamma'$ are derived equivalent.
\end{theorem}
In particular, if $d=3$, then all NCCR of $\Lambda$ are derived equivalent.

We call an $R$-order $\Lambda$ \emph{symmetric} if $\Hom_R(\Lambda,R)\simeq \Lambda$ as $(\Lambda,\Lambda)$-modules.
If $d=3$ and $\Lambda$ is a symmetric $R$-order which is an isolated singularity, then we have the following nice behaviour of
modules giving NCCR of $\Lambda$ \cite{IR}.
\begin{theorem}\label{CT and NCCR}
Let $d=3$ and $\Lambda$ be a symmetric $R$-order which is an isolated singularity.
\begin{itemize}
\item[(a)] $\CM(\Lambda)$ contains a 2-cluster tilting object if and only if $\Lambda$ has a NCCR.
\item[(b)] Let $M$ be a 2-cluster tilting object in $\CM(\Lambda)$.
Then $N\mapsto\Hom_\Lambda(M,N)$ gives a bijection between $\Lambda$-modules giving NCCR and tilting $\End_\Lambda(M)$-modules which are reflexive $R$-modules.
\end{itemize}
\end{theorem}

\subsection{Cohen-Macaulay modules and triangulated categories}\label{Cohen-Macaulay modules and triangulated categories}

In this subsection we observe that nice triangulated categories are obtained from the category of Cohen-Macaulay modules.

Again, let $R$ be a complete local Gorenstein ring of Krull dimension $d$.
We call an $R$-order $\Lambda$ \emph{Gorenstein} if $D_d\Lambda$ is a projective $\Lambda$-module, or equivalently $D_d\Lambda$ is a projective $\Lambda^{\rm op}$-module.
For example, Gorenstein $R$-orders of Krull dimension $d=0$ are nothing but selfinjective artin $R$-algebras.
In this case, $\CM(\Lambda)$ forms a Frobenius category whose projective objects are just projective $\Lambda$-modules.
By a general result of Happel \cite{Ha}, the stable category $\underline{\CM}(\Lambda)$ forms a triangulated category with a suspension functor $\Omega^{-1}$.
Let us observe that the category $\underline{\CM}(\Lambda)$ enjoys a further nice property.

\begin{definition}\label{serre functor}
Let $\TTT$ be an $R$-linear category such that $\Hom_{\TTT}(X,Y)\in\fl R$ for any $X,Y\in\TTT$.
We call an autoequivalence $F$ of $\TTT$ a \emph{Serre functor} if there exists a functorial isomorphism
\[\Hom_{\TTT}(X,Y)\simeq D\Hom_{\TTT}(Y,FX)\]
for any $X,Y\in\TTT$ \cite{BK,RV2}.
A triangulated category $\TTT$ with the suspension functor $\Sigma$ is called \emph{$n$-Calabi-Yau} (or \emph{$n$-CY}) for an integer $n$
if $\Sigma^n$ gives a Serre functor of $\TTT$ \cite{Ko,Ke}.

A Frobenius category $\XX$ is called \emph{$n$-CY} if the stable category $\TTT:=\underline{\XX}$ is an $n$-CY triangulated category.
\end{definition}
We have the following result from the Auslander-Reiten duality given in Theorem \ref{AR theory for orders}.
\begin{theorem}
Let $\Lambda$ be a Gorenstein $R$-order which is an isolated singularity.
Then $\underline{\CM}(\Lambda)$ forms a triangulated category with a Serre functor $\Omega^{-1}\tau$.
\end{theorem}
A special class of Gorenstein $R$-orders is given by symmetric $R$-orders defined in the previous subsection.
For a symmetric $R$-order $\Lambda$, we have that the Nakayama functor $\nu:\CM(\Lambda)\to\CM(\Lambda)$ is isomorphic to the identity functor $\id_{\CM(\Lambda)}$.
Using this property one can show the following result due to Auslander \cite{A-order}.
\begin{proposition}\label{CM is (d-1)-Calabi-Yau}
Let $\Lambda$ be a symmetric $R$-order which is an isolated singularity.
Then we have an isomorphism $\tau\simeq\Omega^{2-d}$ of functors $\underline{\CM}(\Lambda)\to\underline{\CM}(\Lambda)$.
In particular $\underline{\CM}(\Lambda)$ forms a $(d-1)$-CY triangulated category.
\end{proposition}
For example, if $\Lambda$ is a finite-dimensional symmetric algebra over a field $k$, then the stable category $\underline{\mod}\Lambda$ forms a $(-1)$-CY triangulated category.

In the rest of this subsection we discuss other triangulated categories with Serre functors.
Let $\Lambda$ be a module-finite $R$-algebra.
We denote by $\DD^{\rm b}(\mod\Lambda)$ (respectively, $\DD^{\rm b}(\fl\Lambda)$) the bounded derived category of $\mod\Lambda$ (respectively, $\fl\Lambda$),
and by $\KK^{\rm b}(\add\Lambda)$ the homotopy category of bounded complexes on $\add\Lambda$.
It is not difficult to check the following result \cite{Ha,IY,AH}.
\begin{proposition}\label{serre functor of derived category}
\begin{itemize}
\item[(a)] If $\gl\Lambda<\infty$, then $\DD^{\rm b}(\fl\Lambda)$ forms a triangulated category with a Serre functor $-\stackrel{\bf L}{\otimes}_\Lambda{\bf R}\Hom_R(\Lambda,R)[d]$.
\item[(b)] If $\id{}_\Lambda\Lambda=\id\Lambda_\Lambda<\infty$, then
$\DD^{\rm b}(\fl\Lambda)\cap\KK^{\rm b}(\add\Lambda)$ forms a triangulated category with a Serre functor $-\stackrel{\bf L}{\otimes}_\Lambda{\bf R}\Hom_R(\Lambda,R)[d]$.
\end{itemize}
\end{proposition}
A module-finite $R$-algebra $\Lambda$ is called \emph{$n$-CY} for an integer $n$ if $\DD^{\rm b}(\fl\Lambda)$ forms an $n$-CY triangulated category
in the sense of Definition \ref{serre functor} \cite{IR} (see also \cite{CRo,G}).
Similarly, $\Lambda$ is called \emph{perfectly $n$-CY}
if $\DD^{\rm b}(\fl\Lambda)\cap\KK^{\rm b}(\add\Lambda)$ forms an $n$-CY triangulated category.
Immediately from Proposition \ref{serre functor of derived category} we have that
\begin{itemize}
\item any non-singular symmetric $R$-order is $d$-CY,
\item any symmetric $R$-order is perfectly $d$-CY.
\end{itemize}
Notice that a symmetric $R$-order $\Lambda$ is non-singular if and only if $\gl\Lambda<\infty$.
For example, a finite-dimensional symmetric algebra $\Lambda$ over a field is perfectly $0$-CY.
If a symmetric $R$-order $\Lambda$ has a NCCR $\Gamma$, then it is not difficult to show that
$\Gamma$ is a non-singular symmetric $R$-order, hence $\Gamma$ is $d$-CY.

For example, let $k$ be a field of characteristic zero and $S=k[[x_1,\cdots,x_d]]$ the formal power series ring.
For a finite subgroup $G$ of $\SL_d(k)$, let $\Lambda$ be the invariant subring $S^G$.
Then $S$ gives a NCCR $\End_\Lambda(S)\simeq S*G$ of $\Lambda$ \cite{V2}, and $S*G$ is a $d$-CY algebra.

In particular, we have the following observation.
\begin{proposition}\label{(d-1) and d}
Let $\Lambda$ be a symmetric $R$-order which is an isolated singularity.
If $M$ is a $(d-1)$-cluster tilting object in $\CM(\Lambda)$, then $\End_\Lambda(M)$ is a non-singular symmetric $R$-order, and hence a $d$-CY algebra.
\end{proposition}

We end this subsection by the following characterization of $n$-CY algebras \cite{IR}.
\begin{theorem}
Let $\Lambda$ be a module-finite $R$-algebra which is a faithful $R$-module.
Then the following conditions are equivalent.
\begin{itemize}
\item[(a)] $\Lambda$ is $n$-CY for some $n$.
\item[(b)] $\Lambda$ is $d$-CY.
\item[(c)] $\Lambda$ is a non-singular symmetric $R$-order.
\end{itemize}
\end{theorem}

\subsection{$2$-cluster tilting for $2$-Calabi-Yau categories}\label{2-cluster tilting for 2-Calabi-Yau categories}

Let $\TTT$ be a triangulated category.
We define $n$-cluster tilting subcategories (respectively, objects) and $n$-rigid subcategories (respectively, objects) of $\TTT$
by replacing $\Ext^i_{\A}(-,-)$ in Definition \ref{definition of n-cluster tilting} by $\Hom_{\TTT}(-,\Sigma^i-)$.

In this section we briefly recall a few results on $2$-cluster tilting subcategories of $2$-CY triangulated categories.
They are mainly studied from the viewpoint of categorification of Fomin-Zelevinsky cluster algebras \cite{FZ}
using cluster categories \cite{BMRRT,CCS,CK} and stable categories of preprojective algebras \cite{GLS1,GLS5,BIRSc}.
Especially, the following unique replacement property \cite{BMRRT,IY} of indecomposable direct summands of $2$-cluster tilting objects is important.
\begin{theorem}
Let $M$ be a basic $2$-cluster tilting object in a 2-CY triangulated category $\TTT$.
For any indecomposable direct summand $X$ of $M$, we take a decomposition $M=X\oplus N$.
\begin{itemize}
\item[(a)] There exists exactly one indecomposable object $Y\in\TTT$ which is not isomorphic to $X$ such that $Y\oplus N$ is a basic $2$-cluster tilting object in $\TTT$.
\item[(b)] There exist triangles
\[X\xrightarrow{g'}N_1\xrightarrow{f'}Y\to\Sigma X\ \mbox{ and }\
Y\xrightarrow{g}N_0\xrightarrow{f}X\to\Sigma Y\]
in $\TTT$ such that $f$ and $f'$ are right $(\add N)$-approximation and $g$ and $g'$ are left $(\add N)$-approximation.
\end{itemize}
\end{theorem}
In this case, we call $Y\oplus N$ a \emph{$2$-cluster tilting mutation} of $X\oplus N$.
This is an analogue of tilting mutations which we shall apply in Subsection \ref{Hypersurface singularities}.
There are many interesting results on $2$-cluster tilting mutations.
For example, the quivers of $\End_{\TTT}(X\oplus N)$ and $\End_{\TTT}(Y\oplus N)$ are related by Fomin-Zelevinsky mutation \cite{FZ}, \cite{BMR2,BIRSc}.

Next we consider the triangulated analogue of $2$-Auslander algebras.
The endomorphism algebra $\End_{\TTT}(M)$ of a $2$-cluster tilting object $M$ in a 2-CY triangulated category $\TTT$ is called a \emph{2-CY tilted algebra}.
When $\TTT$ is the stable category $\underline{\CM}(\Lambda)$ of an $R$-order $\Lambda$, then $\End_{\TTT}(M)$ is
a factor algebra of the $2$-Auslander algebra $\End_\Lambda(M)$ which enjoys nice properties (see Theorem \ref{Auslander algebra for n=d-1}, Proposition \ref{(d-1) and d}).
The following result \cite{BMR,KR} shows that 2-CY tilted algebras also enjoy nice properties.
\begin{theorem} \label{properties of 2-CY}
Let $\Gamma=\End_{\TTT}(M)$ be a 2-CY tilted algebra.
\begin{itemize}
\item[(a)] $\id{}_\Gamma\Gamma=\id\Gamma_\Gamma\le 1$ holds.
\item[(b)] The following category forms a $3$-CY Frobenius category
\[\Sub\Gamma:=\{X\in\mod\Gamma\ |\ X \mbox{ is a submodule of $\Gamma^\ell$ for some }\ell\}.\]
\item[(c)] We have the following equivalence of categories
\[\Hom_{\TTT}(M,-):\TTT/[\Sigma M]\to\mod\Gamma.\]
\end{itemize}
\end{theorem}
We end this subsection with the following analogue \cite{BIRSc} of Bongartz completion for tilting modules of projective dimension at most one \cite{ASS}.
\begin{theorem}\label{cluster tilting completion}
Let $\XX$ be a 2-CY Frobenius category. Assume that $\XX$ has a $2$-cluster tilting object.
Then any $2$-rigid object in $\XX$ is a direct summand of some $2$-cluster tilting object in $\XX$.
\end{theorem}

\section{Examples of $2$-cluster tilting objects}

In this section we give two classes of $2$-CY Frobenius categories with $2$-cluster tilting objects.
One class is constructed from finite-dimensional factor algebras of preprojective algebras of arbitrary quivers without loops.
Another class is constructed from one-dimensional hypersurface singularities $k[[x,y]]/(f)$.
Although these constructions are rather different, they are similar in the sense that we use factor algebras of $2$-CY algebras (preprojective algebras or $k[[x,y]]$).

\subsection{Preprojective algebras}\label{Cluster tilting for preprojective algebras}

In this subsection we shall recall the construction of 2-CY Frobenius categories with $2$-cluster tilting objects given in joint work with Buan, Reiten and Scott \cite{IR,BIRSc}.
This is closely related to the work of Geiss, Leclerc and Schr\"oer \cite{GLS4,GLS5}.

To accord with the conventions in \cite{BIRSc}, \emph{all modules in this subsection are left modules}.
Throughout this subsection we fix a connected quiver $Q$ without loops.
Let $\widetilde{Q}$ be the quiver constructed from $Q$ by adding an arrow $a^*:j\to i$ for each arrow $a:i\to j$ in $Q$.
We denote by $\widehat{k\widetilde{Q}}$ the \emph{complete path algebra} of $\widetilde{Q}$ over a field $k$.
Thus as a $k$-vector space $\widehat{k\widetilde{Q}}$ is a direct product $\prod_{i\ge0}k\widetilde{Q}_i$
of $k$-vector spaces $k\widetilde{Q}_i$ whose basis is given by all paths in $\widetilde{Q}$ of length $i$,
and the multiplication in $\widehat{k\widetilde{Q}}$ is given by the same way as for usual path algebra $k\widetilde{Q}$.
The \emph{complete preprojective algebra} of $Q$ is defined as $\Lambda:=\widehat{k\widetilde{Q}}/I$, where $I$ is the closure of the ideal
\[\langle\sum_{a\in Q_1}(aa^*-a^*a)\rangle\]
of $\widehat{k\widetilde{Q}}$ with respect to the $J_{\widehat{k\widetilde{Q}}}$-adic topology on $\widehat{k\widetilde{Q}}$.
For the case of Dynkin and extended Dynkin quivers, we have the following.
\begin{itemize}
\item $Q$ is a Dynkin quiver if and only if $\Lambda$ is finite-dimensional over $k$. In this case $\Lambda$ forms a selfinjective algebra.
\item If $Q$ is an extended Dynkin quiver, then $\Lambda$ is Morita-equivalent to the skew group algebra $k[[x,y]]*G$ for a finite subgroup $G$ of $\SL_2(k)$.
\end{itemize}
We observed in Subsection \ref{Cohen-Macaulay modules and triangulated categories} that $k[[x,y]]*G$ is a 2-CY algebra. More generally, we have the following result \cite{CB2,CRo,GLS3}.
\begin{theorem}\label{preprojective is 2-CY}
\begin{itemize}
\item[(a)] If $Q$ is a Dynkin quiver, then $\mod\Lambda$ forms a 2-CY Frobenius category.
\item[(b)] If $Q$ is not a Dynkin quiver, then $\DD^{\rm b}(\fl\Lambda)$ forms a 2-CY triangulated category.
\end{itemize}
\end{theorem}
Let $Q_0=\{1,\cdots,n\}$ be the set of vertices of $Q$,
and we denote by $e_i$ the idempotent of $\Lambda$ corresponding to $i\in Q_0$.
Define a 2-sided ideal of $\Lambda$ by
\[I_i:=\Lambda(1-e_i)\Lambda.\]
Since $I_i$ is generated by an idempotent, we have
\begin{itemize}
\item[(i)] $I_i^2=I_i$ for any $i$.
\end{itemize}
We denote by
\[\langle I_1,\cdots,I_n\rangle\]
the set of two-sided ideals $I_{i_1} \cdots I_{i_l}$, $i_1,\dots,i_l \in \{1,\dots,n\}$,
of $\Lambda$ obtained by multiplying the ideals $I_1,\cdots,I_n$.
This set contains $\Lambda$, too.
We call a left ideal $I$ of $\Lambda$ \emph{cofinite tilting} if $\Lambda/I\in\fl\Lambda$ and $I$ is a tilting $\Lambda$-module.
Dually, we define \emph{cofinite tilting} right ideals.
If $Q$ is a Dynkin quiver, then any tilting $\Lambda$-module is projective since $\Lambda$ is selfinjective.
Otherwise, we have a lot of tilting $\Lambda$-modules by the following result.
\begin{theorem}\label{tilting over 2-CY}
Assume that $Q$ is not a Dynkin quiver.
\begin{itemize}
\item[(a)] $\langle I_1,\cdots,I_n\rangle$ gives a set of cofinite tilting left (respectively, right) ideals of $\Lambda$.
They have projective dimension at most one, and two different elements are non-isomorphic as left (respectively, right) $\Lambda$-modules.

\item[(b)] If $Q$ is an extended Dynkin quiver, then $\langle I_1,\cdots,I_n\rangle$ gives a set of isoclasses of basic tilting $\Lambda$-modules (respectively, tilting $\Lambda^{\rm op}$-modules) of projective dimension at most one.
\end{itemize}
\end{theorem}
For $i,j\in\{1,\cdots,n\}$ with $i\neq j$, we define the two-sided ideal of $\Lambda$ by $I_{i,j}:=\Lambda(1-e_i-e_j)\Lambda$.
Then any ideal of $\Lambda$ obtained by multiplying $I_i$ and $I_j$ any times contains $I_{i,j}$.
If there is no arrow between $i$ and $j$, then the factor algebra $\Lambda/I_{i,j}$ is isomorphic to $k\times k$. This implies
\begin{itemize}
\item[(ii)] $I_iI_j=I_{i,j}=I_jI_i$, if there is no arrow in $Q$ between $i$ and $j$.
\end{itemize}
If there is exactly one arrow in $Q$ between $i$ and $j$, then $\Lambda/I_{i,j}$ is isomorphic to the preprojective algebra of type $A_2$. Looking at Loewy series of $\Lambda$, we have
\begin{itemize}
\item[(iii)] $I_iI_jI_i=I_{i,j}=I_jI_iI_j$, if there is exactly one arrow in $Q$ between $i$ and $j$.
\end{itemize}
Using the relations (i), (ii) and (iii), we can describe the set $\langle I_1,\cdots,I_n\rangle$.
\begin{definition}
The \emph{Coxeter group} $W$ \cite{BB} of $Q$ is a group presented by generators $s_1,\cdots,s_n$ with relations
\begin{itemize}
\item[(a)] $s_i^2=1$, for any $i$;
\item[(b)] $s_is_j=s_js_i$, if there is no arrow in $Q$ between $i$ and $j$;
\item[(c)] $s_is_js_i=s_js_is_j$, if there is exactly one arrow in $Q$ between $i$ and $j$.
\end{itemize}
\end{definition}
The relations (ii) and (iii) coincide with (b) and (c), respectively, but (i) is slightly different from (a).
Thus it is natural to compare the set $\langle I_1,\cdots,I_n\rangle$ with $W$.
In fact, we have the following result.
\begin{theorem}\label{coxeter}
Let $Q$ be an arbitrary quiver without loops.
Then there exists a well-defined bijection $W\xrightarrow{\sim}\langle I_1,\cdots,I_n\rangle$.
It is given by $w\mapsto I_w:=I_{i_1}\cdots I_{i_\ell}$ for an arbitrary reduced expression $w=s_{i_1}\cdots s_{i_\ell}$ of $w$ in $W$.
\end{theorem}
As an immediate consequence of Theorems \ref{tilting over 2-CY} and \ref{coxeter}, we have the following result
which is related to work of Ishii and Uehara \cite{IU} on McKay correspondence for the minimal resolutions of simple singularities.
\begin{corollary}\label{extended dynkin case}
Let $Q$ be an extended Dynkin quiver. Then basic tilting $\Lambda$-modules correspond bijectively to elements in $W$.
\end{corollary}
Again, let $Q$ be an arbitrary quiver. For any $w\in W$, we put
\[\Lambda_w:=\Lambda/I_w,\]
which is a finite-dimensional algebra. It can be shown that $\Lambda_w$ satisfies
\[\id_{\Lambda_w}(\Lambda_w)=\id(\Lambda_w)_{\Lambda_w}\le 1.\]
Thus $\Lambda_w$ is a cotilting $\Lambda_w$-module of injective dimension at most one, and the classical tilting theory implies that the full subcategory
\[\Sub\Lambda_w:=\{X\in\fl\Lambda\ |\ X \mbox{ is a submodule of $\Lambda_w^\ell$ for some }\ell\}\]
of $\fl\Lambda$ forms an extension-closed subcategory of $\fl\Lambda$ with enough projectives and enough injectives $\add\Lambda_w$.
In particular $\Sub\Lambda_w$ forms a Frobenius category, and $\Ext^1_{\Lambda_w}(X,Y)$ and $\Ext^1_\Lambda(X,Y)$ are isomorphic for any $X,Y\in\Sub\Lambda_w$.
By Theorem \ref{preprojective is 2-CY}, we have a functorial isomorphism
\[\Ext^1_{\Lambda_w}(X,Y)\simeq D\Ext^1_{\Lambda_w}(Y,X)\]
for any $X,Y\in\Sub\Lambda_w$.
Consequently $\Sub\Lambda_w$ forms a 2-CY Frobenius category.
Moreover, it contains $2$-cluster tilting objects by the following result.
\begin{theorem}\label{cluster tilting from reduced expression}
Let $w\in W$. For any reduced expression $w=s_{i_1}\cdots s_{i_\ell}$ of $w$ in $W$, we have the following $2$-cluster tilting object in the 2-CY Frobenius category $\Sub\Lambda_w$
\[T(i_1,\cdots,i_\ell):=\bigoplus_{k=1}^\ell\Lambda_{s_{i_1}\cdots s_{i_k}}=\bigoplus_{k=1}^\ell\Lambda/I_{i_1}\cdots I_{i_k}.\]
\end{theorem}
A crucial step in the proof of this theorem is to show that the global dimension of the endomorphism algebra
$\End_{\Lambda_w}(T(i_1\cdots,i_\ell))$ is at most three by using induction on $\ell$.
Then we can apply Lemma \ref{criterion for n-cluster tilting 2}.

It is easy to show that $\add T(i_1,\cdots,i_\ell)$ contains exactly one indecomposable object which is not contained in $\add T(i_1,\cdots,i_{\ell-1})$.
It is shown in \cite{BIRSc} that any full subcategory of $\fl\Lambda$ which is closed under extensions and submodules and contains $2$-cluster tilting objects has the form $\Sub\Lambda_w$ for some $w\in W$.

We have a nice description of the quivers of the endomorphism algebras of these $2$-cluster tilting objects.
For each reduced expression $w=s_{i_1}\cdots s_{i_\ell}$, we define a quiver $Q(i_1,\cdots,i_\ell)$ as follows:
\begin{itemize}
\item For each $i\in\{1,\cdots,n\}$, pick out the expression consisting of the $i_k$ which are $i$.
We draw an arrow from each $i$ to the previous $i$.
\item For each arrow $a:i\to j$ in $\widetilde{Q}$, pick out the expression consisting of the $i_k$ which are $i$ or $j$.
We draw an arrow from the last $i$ in a connected set of $i$'s to the last $j$ in the next set of $j$'s
\end{itemize}
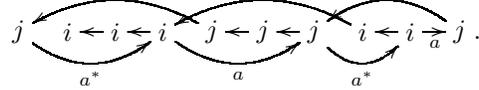
\[\xymatrix@R0.1cm@C0.3cm{
j\ar@/_1pc/[rrr]_{a^*}&i&i\ar[l]&i\ar[l]\ar@/_1pc/[rrr]_a&j\ar@/_1pc/[llll]&j\ar[l]&j\ar[l]\ar@/_1pc/[rr]_{a^*}&i\ar@/_1pc/[llll]&i\ar[l]\ar[r]_a&j\ar@/_1pc/[lll]} .\]
Define the quiver $\underline{Q}(i_1,\cdots,i_\ell)$ by removing the last $i$ from $Q(i_1,\cdots,i_\ell)$ for each $i\in\{1,\cdots,n\}$.
\begin{theorem}
Let $w=s_{i_1}\cdots s_{i_\ell}$ be a reduced expression, and let $T:=T(i_1,\cdots,i_\ell)$ be the corresponding $2$-cluster tilting object in $\Sub\Lambda_w$.
Then the quiver of $\End_\Lambda(T)$ is $Q(i_1\cdots,i_\ell)$, and the quiver of $\underline{\End}_\Lambda(T)$ is $\underline{Q}(i_1\cdots,i_\ell)$.
\end{theorem}
Moreover, it is shown in \cite{BIRSm} that the algebra $\underline{\End}_\Lambda(T)$ is isomorphic to the Jacobian algebra of a quiver with a potential \cite{DWZ}.

\begin{example} \label{example-E}
Let $Q$ be the quiver $\xymatrix@R0.1cm@C0.3cm{1\ar[r]&2\ar[r]&3\ar@/_1pc/[ll]}$,
and let $w=s_1s_2s_1s_3s_2=s_2s_1s_2s_3s_2=s_2s_1s_3s_2s_3$ be reduced expressions.
The first expression gives the quiver
$$
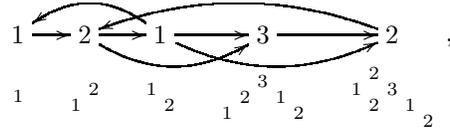
\xymatrix@R0.1cm@C0.3cm{1\ar[r]&2\ar[r]\ar@/_1pc/[rr]&
1\ar@/_1pc/[ll]\ar@/_1pc/[rr]\ar[r]&3\ar[r]&2\ar@/_1pc/[lll]\\
{\begin{smallmatrix}1\end{smallmatrix}}&
{\begin{smallmatrix}&2\\ 1&\end{smallmatrix}}&
{\begin{smallmatrix}1&\\ &2\end{smallmatrix}}&
{\begin{smallmatrix}&&3&&\\ &2&&1&\\ 1&&&&2\end{smallmatrix}}&
{\begin{smallmatrix}&2&&&\\ 1&&3&&\\ &2&&1&\\ &&&&2\end{smallmatrix}}},$$
the second one gives the quiver
$$
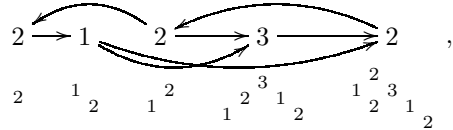
\xymatrix@R0.1cm@C0.3cm{2\ar[r]&1\ar@/_1pc/[rrr]\ar@/_1pc/[rr]&
2\ar@/_1pc/[ll]\ar[r]&3\ar[r]&2\ar@/_1pc/[ll]\\
{\begin{smallmatrix}2\end{smallmatrix}}&
{\begin{smallmatrix}1&\\ &2\end{smallmatrix}}&
{\begin{smallmatrix}&2\\ 1&\end{smallmatrix}}&
{\begin{smallmatrix}&&3&&\\ &2&&1&\\ 1&&&&2\end{smallmatrix}}&
{\begin{smallmatrix}&2&&&\\ 1&&3&&\\ &2&&1&\\ &&&&2\end{smallmatrix}}},$$
and the third one gives the quiver
$$
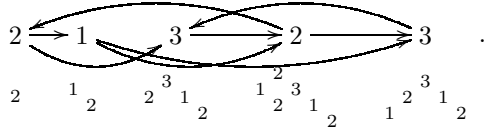
\xymatrix@R0.1cm@C0.3cm{2\ar[r]\ar@/_1pc/[rr]&1\ar@/_1pc/[rr]\ar@/_1pc/[rrr]&
3\ar[r]&2\ar@/_1pc/[lll]\ar[r]&3\ar@/_1pc/[ll]\\
{\begin{smallmatrix}2\end{smallmatrix}}&
{\begin{smallmatrix}1&\\ &2\end{smallmatrix}}&
{\begin{smallmatrix}&3&&\\ 2&&1&\\ &&&2\end{smallmatrix}}&
{\begin{smallmatrix}&2&&&\\ 1&&3&&\\ &2&&1&\\ &&&&2\end{smallmatrix}}&
{\begin{smallmatrix}&&3&&\\ &2&&1&\\ 1&&&&2\end{smallmatrix}}}.$$
\end{example}

We end this subsection with the following question.
\begin{itemize}
\item Do Theorem \ref{tilting over 2-CY}(b) and Corollary \ref{extended dynkin case} hold for arbitrary quivers which are neither Dynkin nor extended Dynkin?
\end{itemize}

\subsection{Hypersurface singularities}\label{Hypersurface singularities}

In this subsection we study $n$-cluster tilting for isolated hypersurface singularities.
Especially, we recall a classification of 2-cluster tilting objects for one-dimensional hypersurface singularity given in a joint work with Burban, Keller and Reiten \cite{BIKR}.

Let $k$ be a field, and $S=k[[x_0,\cdots,x_d]]$ the formal power series ring of $(d+1)$ variables.
We put
\[\Lambda=S/(f)\]
for $f\in(x_0,\cdots,x_d)$.
Then $\Lambda$ is a complete local Gorenstein ring of Krull dimension $d$.
In this subsection we assume that $\Lambda$ is an isolated singularity. Then the category $\CM(\Lambda)$ of Cohen-Macaulay $\Lambda$-modules
forms a $(d-1)$-CY Frobenius category by Proposition \ref{CM is (d-1)-Calabi-Yau}.

A remarkable property of hypersurface singularities is given by \emph{matrix factorization} introduced by Eisenbud \cite{Ei,Y}.
For any $X\in\CM(\Lambda)$, we have $\pd X_S=1$ by Auslander-Buchsbaum formula \eqref{Auslander-Buchsbaum}.
Thus there exists a projective resolution
\[0\to S^\ell\xrightarrow{a}S^\ell\to X\to0,\]
where $a$ can be regarded as an $\ell\times\ell$ matrix with entries in $S$.
Using $fX=0$, one can easily check that there exists another $\ell\times\ell$ matrix $b\in\Hom_S(S^\ell,S^\ell)$ such that $ab=f\cdot 1_{S^\ell}$ and $ba=f\cdot 1_{S^\ell}$.
It is easily checked that we have a 2-periodic projective resolution
\begin{equation}\label{2-periodic of X}
\cdots\xrightarrow{b}\Lambda^\ell\xrightarrow{a}\Lambda^\ell\xrightarrow{b}\Lambda^\ell\xrightarrow{a}\Lambda^\ell\to X\to0
\end{equation}
of the $\Lambda$-module $X$. In particular, we have the following conclusion.
\begin{theorem}\label{2-periodic}
We have $\Omega^2\simeq\id_{\underline{\CM}(\Lambda)}$ as functors $\underline{\CM}(\Lambda)\to\underline{\CM}(\Lambda)$.
\end{theorem}
Immediately we have the following observation.
\begin{itemize}
\item If $d$ is even, then $\CM(\Lambda)$ forms a 1-CY Frobenius category. In particular, any $2$-rigid object in $\CM(\Lambda)$ is projective.

\item If $d$ is odd, then $\CM(\Lambda)$ forms a 2-CY Frobenius category. In particular, any $3$-rigid object in $\CM(\Lambda)$ is projective.
\end{itemize}
Another important application of matrix factorization is the following equivalence often called \emph{Kn\"orrer periodicity} \cite{Kn,Y}
(see \cite{So} for characteristic two).
\begin{theorem}\label{periodicity}
Let $\Lambda':=k[[x_0,\cdots,x_d,u,v]]/(f+uv)$. Then we have a triangulated equivalence
\[\underline{\CM}(\Lambda)\to\underline{\CM}(\Lambda'),\]
which sends $X\in\CM(\Lambda)$ with the projective resolution \eqref{2-periodic of X} to $X'\in\CM(\Lambda')$ with the following projective resolution
\[\cdots\xrightarrow{{b\ \ u\choose v\ -a}}\Lambda'{}^{2\ell}\xrightarrow{{a\ \ u\choose v\ -b}}\Lambda'{}^{2\ell}
\xrightarrow{{b\ \ u\choose v\ -a}}\Lambda'{}^{2\ell}\xrightarrow{{a\ \ u\choose v\ -b}}\Lambda'{}^{2\ell}\to X'\to0.\]
\end{theorem}
In the rest of this subsection we concentrate on the case $d=1$ and study $2$-cluster tilting objects of $\CM(\Lambda)$.
Thus
\[S=k[[x,y]]\ \mbox{ and }\ \Lambda=S/(f).\]
In this case, $\CM(\Lambda)$ consists of $X\in\mod\Lambda$ such that $\soc X=0$.
For a technical reason we assume that $k$ is an infinite field.
We take a decomposition $f=f_1\cdots f_n$ into irreducible factors.
Since we assumed that $\Lambda$ is an isolated singularity, we have that $(f_i)\neq(f_j)$ for any $i\neq j$.
We put
\[S_i:=S/(f_1\cdots f_i)\]
for any $1\le i\le n$.
We have a projective resolution
\[\cdots\xrightarrow{f_1\cdots f_i}\Lambda\xrightarrow{f_{i+1}\cdots f_n}\Lambda\xrightarrow{f_1\cdots f_i}\Lambda\to S_i\to0\]
of $S_i$ as the simplest case of matrix factorization. Applying $\Hom_\Lambda(-,S_j)$, one can easily check that $\Ext^1_\Lambda(S_i,S_j)=0$ holds for any $1\le i,j\le n$.
We have the following criterion for existence of $2$-cluster tilting objects in $\CM(\Lambda)$.
\begin{theorem}\label{existence of cluster tilting}
If $f_i\notin(x,y)^2$ for any $i$, then $\CM(\Lambda)$ has a 2-cluster tilting object $M:=\bigoplus_{i=1}^nS_i$. The converse holds if $k$ is an algebraically closed field of characteristic zero.
\end{theorem}
We give examples for representation-finite cases \cite{Y}.

\begin{example} \label{example-F}
Let $\Lambda=k[[x,y]]/(f)$ be a simple singularity, so in characteristic zero $f$ is one of the following
\[\begin{array}{clc}
(A_m)&x^2+y^{m+1}&(m\ge 1),\\
(D_m)&x^2y+y^{m-1}&(m\ge4),\\
(E_6)&x^3+y^4,&\\
(E_7)&x^3+xy^3,&\\
(E_8)&x^3+y^5.&
\end{array}\]
By our theorem $\CM(\Lambda)$ has a 2-cluster tilting object if and only if $\Lambda$ is of type $A_m$ with odd $m$ or $D_m$ with even $m$.
\end{example}

We give a sketch of the proof of Theorem \ref{existence of cluster tilting}.

Let us prove the former assertion.
For simplicity we assume that $(f_i,f_{i+1})=(x,y)$ holds for any $1\le i<n$.
(Any case can be reduced to this case by the assumption that $k$ is an infinite field.)
Since we already observed that $M$ is a 2-rigid generator-cogenerator in $\CM(\Lambda)$,
we only have to construct the sequence in Lemma \ref{criterion for n-cluster tilting 2}(c) for each $X=S_i$ ($1\le i\le n$).
It is given by
\begin{eqnarray*}
&0\to S_i\xrightarrow{{f_{i+1}\choose -1}}S_{i+1}\oplus S_{i-1}
\xrightarrow{{f_i\ f_if_{i+1}\choose 1\ \ \ f_{i+1}}}S_{i+1}\oplus
S_{i-1}\xrightarrow{(-1\ f_i)}S_i\to 0\ \ (1\le i<n),&\\
&0\to S_{n-1}\xrightarrow{{f_n\choose -f_{n+1}}}S_n\oplus S_{n-1}
\xrightarrow{(f_{n+1}\ f_n)}S_n\ \ (i=n).&
\end{eqnarray*}
Notice that the upper sequence gives a 2-almost split sequence in $\add M$.

The proof of the converse is indirect and needs results in birational geometry. Thus our assumptions on the base field $k$ are necessary.
Let \[\Lambda':=k[[x,y,u,v]]/(f(x,y)+uv).\]
By Theorem \ref{periodicity}, the following conditions are equivalent.
\begin{itemize}
\item[(i)] $\CM(\Lambda)$ has a $2$-cluster tilting object.
\item[(ii)] $\CM(\Lambda')$ has a $2$-cluster tilting object.
\end{itemize}
By Theorem \ref{CT and NCCR}, the condition (ii) is equivalent to the following condition.
\begin{itemize}
\item[(iii)] $\Lambda'$ has a non-comutative crepant resolution.
\end{itemize}
Now let us translate the condition (iii) into a geometric condition.
The singularity $\Lambda'$ is a $cA_m$-singularity for $m:=\ord(f)-1$ in the sense that
the intersection of the hypersurface $f(x,y)+uv=0$ with a generic hyperplane $ax+by+cu+dv=0$ in $k^4$ is an $A_m$-singularity.
Since any isolated $cA_m$-singularity is a terminal singularity \cite{Re}, the condition (iii) is equivalent to the following condition by results of Van den Bergh \cite{V1,V2}.
\begin{itemize}
\item[(iv)] $\Spec(\Lambda')$ has a crepant resolution.
\end{itemize}
For isolated $cA_m$-singularities, the condition (iv) is equivalent to the following condition by a result of Katz \cite{Ka}.
\begin{itemize}
\item[(v)] The number of irreducible factors of $f$ is $m+1$.
\end{itemize}
Clearly (v) is equivalent to that $f_i\notin(x,y)^2$ for any $1\le i\le n$.
Thus we finish the proof of the converse.
Notice that we have also proved the former assertion by a different approach without giving explicit $2$-cluster tilting objects in $\CM(\Lambda)$.
\qed

\medskip
To state our classification result of $2$-cluster tilting objects in $\CM(\Lambda)$, we introduce some notations.
We denote by ${\mathfrak S}_n$ the symmetric group of degree $n$. For $w\in{\mathfrak S}_n$, we put
\[S_i^w:=S/(f_{w(1)}f_{w(2)}\cdots f_{w(n)})\ \mbox{ and }\ M_w:=\bigoplus_{i=1}^nS_i^w.\]
For a non-empty subset $I$ of $\{1,\cdots,n\}$, we put
\[S_I:=S/(\prod_{i\in I}f_i).\]
\begin{theorem}\label{n! cluster tilting}
Assume that $f_i\notin(x,y)^2$ for any $i$.
\begin{itemize}
\item[(a)] There exist exactly $n!$ basic $2$-cluster tilting objects $M_w$ ($w\in{\mathfrak S}_n$) in $\CM(\Lambda)$.
\item[(b)] There exist exactly $(2^n-1)$ indecomposable rigid objects $S_I$ ($I\subset\{1,\cdots,n\}$, $I\neq\emptyset$) in $\CM(\Lambda)$.
\end{itemize}
\end{theorem}
We shall give a sketch of proof.
Thanks to Theorem \ref{cluster tilting completion}, we only have to show the part (a).
Fix $w\in{\mathfrak S}_n$ and put $\Gamma:=\End_\Lambda(M_w)$.
Then $\Gamma$ forms a $\Lambda$-order and we have a functor
\[P_-:=\Hom_\Lambda(M_w,-):\CM(\Lambda)\to\CM(\Gamma).\]
Since $M_w$ contains $\Lambda$ as a direct summand, we have that the functor $P_-$ is fully faithful.
Moreover, Theorem \ref{derived equivalence} shows that the image $P_N$ of any basic $2$-cluster tilting object $N$ in $\CM(\Lambda)$ is a basic tilting $\Gamma$-module of projective dimension at most one.
Thus we only have to classify all basic Cohen-Macaulay tilting $\Gamma$-modules of projective dimension at most one.
It is given by the following result.
\begin{theorem}\label{n! tilting modules}
Let $w\in\mathfrak{S}_n$. Then there exist exactly $n!$ basic Cohen-Macaulay tilting $\End_\Lambda(M_w)$-modules $\Hom_\Lambda(M_w,M_{w'})$ ($w'\in{\mathfrak S}_n$) of projective dimension at most one.
\end{theorem}
It is interesting to compare with Corollary \ref{extended dynkin case}, where a bijection between tilting modules of projective dimension at most one
and elements in an affine Weyl group is given.

The ingredient of the proof of Theorem \ref{n! tilting modules} is the theory of `tilting mutation' developed by Riedtmann-Schofield \cite{RS} and Happel-Unger \cite{HU1,HU2}.
Although they deal with finite-dimensional algebras, their theory is valid also for our $\Gamma$.
Let us recall several results.
\begin{itemize}
\item[(a)] The set of basic tilting modules of projective dimension at most one has a natural partial order defined as follows.
For basic tilting $\Gamma$-modules $T$ and $U$ of projective dimension at most one, we define
\[T\ge U\]
if $\Ext^1_\Gamma(T,U)=0$ for any $i>0$.
Clearly $\Gamma$ is the unique maximal element.
\item[(b)] Let $T=\bigoplus_{i=1}^nT_i$ be a basic tilting $\Gamma$-module of projective dimension at most one.
A \emph{tilting mutation} of $T$ is a tilting $\Gamma$-module $U$ of projective dimension at most one such that
$T$ and $U$ have exactly $(n-1)$ common indecomposable direct summands.
It is shown in \cite{RS} that for any $i$ ($1\le i\le n$), there exists at most one tilting mutation of $T$ obtained by replacing $T_i$.
\item[(c)] Two tilting $\Gamma$-modules $T$ and $U$ of projective dimension at most one are neighbours with respect to the partial order if and only if they are in the relationship of tilting mutation \cite{HU2}.
\item[(d)] If $T>U$, then there exists a sequence $T=T_0>T_1>T_2>\cdots>U$ of tilting $\Gamma$-modules of projective dimension at most one satisfying the following conditions
\begin{itemize}
\item[$\bullet$] $T_{i+1}$ is a tilting mutation of $T_i$,
\item[$\bullet$] either $T_i=U$ for some $i$ or the sequence is infinite.
\end{itemize}
\end{itemize}
The $\Gamma$-module $P_\Lambda$ is projective-injective in $\CM(\Gamma)$ since we have $\Hom_\Lambda(P_\Lambda,\Lambda)=M_w=\Hom_\Lambda(\Lambda,M_w)$.
Consequently, any Cohen-Macaulay tilting $\Gamma$-module $T$ of projective dimension at most one has $P_\Lambda$
as an indecomposable direct summand, since there exists an exact sequence
\[0\to P_\Lambda\to T_0\to T_1\to0\]
with $T_i\in\add T$, which must split.
Thus the above observation (b) implies that
\begin{itemize}
\item any Cohen-Macaulay tilting $\Gamma$-module has at most $(n-1)$ tilting mutations which are Cohen-Macaulay.
\end{itemize}
On the other hand, $M_{w'}$ and $M_{w's_i}$ ($1\le i\le n-1$) have exactly $(n-1)$ common indecomposable direct summands.
Thus $P_{M_{w'}}$ has at least $(n-1)$ tilting mutations $P_{M_{w's_i}}$ with $1\le i\le n-1$.
Consequently, we have that
\begin{itemize}
\item any Cohen-Macaulay tilting $\Gamma$-module of the form $P_{M_{w'}}$ has exactly $(n-1)$ tilting mutations $P_{M_{w's_i}}$ ($1\le i\le n-1$) which are Cohen-Macaulay.
\end{itemize}
Now we can finish our proof of Theorem \ref{n! tilting modules} as follows.
Take any basic tilting $\Gamma$-module $U$ of projective dimension at most one.
Since $\Gamma>U$, we have a sequence
$\Gamma=T_0>T_1>T_2>\cdots>U$ satisfying the properties in (d) above.
Since $\Gamma=P_{M_w}$, the above observation implies that each $T_i$ has the form $T_i=P_{M_{w_i}}$ for some $w_i\in\mathfrak{S}_n$.
Since $T_i\ {\not\simeq}\ T_j$ for any $i\neq j$, we have $w_i\neq w_j$ for any $i\neq j$. Moreover, since ${\mathfrak S}_n$ is a finite group, the sequence must be finite.
Thus $U=T_i=P_{M_{w_i}}$ for some $i$, and the proof is completed.
\qed

\medskip
We give a few remarks on the 2-CY tilted algebra $\underline{\End}_\Lambda(M)$
with $M$ from Theorem \ref{existence of cluster tilting}.
Since $\CM(\Lambda)$ is 2-CY and $\Omega^2\simeq 1_{\underline{\CM}(\Lambda)}$ by Theorem \ref{2-periodic}, we have
\[\underline{\End}_\Lambda(M)\simeq D\Ext^2_\Lambda(M,M)\simeq D\underline{\Hom}_\Lambda(M,M).\]
Thus we have the first assertion in the following result.
\begin{proposition}
Assume that $f_i\notin(x,y)^2$ for any $1\le i\le n$.
Then $\underline{\End}_\Lambda(M)$ is a finite-dimensional symmetric algebra satisfying $\tau^2\simeq\id_{\underline{\mod}(\underline{\End}_\Lambda(M))}$.
Its quiver is
\[\xymatrix@C0.5cm@R0.5cm{
    &S_1\ar@<0.5ex>[r]& S_2\ar@<0.5ex>[l]\ar@<0.5ex>[r]& \cdots\ar@<0.5ex>[l]\ar@<0.5ex>[r]& S_{n-2}\ar@<0.5ex>[l]\ar@<0.5ex>[r]& S_{n-1}\ar@<0.5ex>[l]&}\]
where in addition there is a loop at $S_i$ ($1\le i<n$) if and only if $(f_i,f_{i+1})\neq(x,y)$.
\end{proposition}
This result is interesting from the viewpoint of the following conjecture of Crawley-Boevey on finite-dimensional algebras $\Lambda$ over algebraically closed fields,
where `only if' part is proved in \cite{CB1},
\begin{itemize}
\item $\Lambda$ is representation-tame if and only if all but a finite number of indecomposable modules $X$ of fixed dimension satisfies $\tau X\simeq X$.
\end{itemize}
Since $\Lambda$ is representation-wild for $n>4$ \cite{DG}, 
Theorem \ref{properties of 2-CY}(c) sugests that 
the algebra $\underline{\End}_\Lambda(M)$ should be representation-wild for $n>4$.
Note that there are examples of representation-wild self-injective algebras $\Lambda$ satisfying $\tau^3\simeq\id_{\underline{\mod}\Lambda}$ \cite{AR-DTr}.

\begin{example} \label{example-G}
We shall give examples of 2-CY tilted algebras (with $M$ from Theorem \ref{existence of cluster tilting}).

(a) Let $\Lambda$ be a simple singularity of type $A_{2m-1}$, so \[\Lambda=k[[x,y]]/((x-y^m)(x+y^m)).\]
By Theorem \ref{n! cluster tilting}, there exist exactly two 2-cluster tilting objects in $\CM(\Lambda)$,
and $\underline{\End}_\Lambda(M)$ is isomorphic to $k[x]/(x^m)$.

(b) Let $\Lambda$ be a simple singularity of type $D_{2m}$, so \[\Lambda=k[[x,y]]/((x-y^{m-1})(x+y^{m-1})y).\]
By Theorem \ref{n! cluster tilting}, there exist exactly six 2-cluster tilting objects in $\CM(\Lambda)$,
and $\underline{\End}_\Lambda(M)$ is given by the quiver
\[\xymatrix{
    \cdot \ar@(dl,ul)[]^{\varphi} \ar@<0.4ex>[r]^{\alpha}&
    \cdot\ar@<0.4ex>[l]^{\beta}}\]
with relations
\[\varphi^{m-1}=\alpha\beta,\ \varphi\alpha=\beta\varphi=0.\]

(c) Let $\Lambda$ be a minimally elliptic curve singularity of type $T_{3,2q+2}$ with $q \ge 3$, so \[\Lambda=k[[x,y]]/((x-y^q)(x+y^q)(x-y^2)).\]
By Theorem \ref{n! cluster tilting}, there exist exactly six 2-cluster tilting objects in $\CM(\Lambda)$, and $\underline{\End}_\Lambda(M)$ is given by the quiver
\[\xymatrix{
    \cdot \ar@(dl,ul)[]^{\varphi} \ar@<0.4ex>[r]^{\alpha}&
    \cdot\ar@<0.4ex>[l]^{\beta}  \ar@(dr,ur)[]_{\psi}}\]
with relations
\[\alpha\beta = \varphi^2,\ \beta\alpha=\psi^q,\ \alpha\psi = \varphi\alpha,\  \psi\beta = \beta\varphi.\]

(d) Let $\Lambda$ be a minimally elliptic curve singularity of type $T_{2p+2,2q+2}$ with $p,q \ge 1$ and $(p,q)\neq(1,1)$, so \[\Lambda=S/((x-y^q)(x+y^q)(x^p-y)(x^p+y)).\]
By Theorem \ref{n! cluster tilting}, there exist exactly twenty four 2-cluster tilting objects in $\CM(\Lambda)$, and $\underline{\End}_\Lambda(M)$ is given by the quiver
\[\xymatrix{
    \cdot \ar@(dl,ul)[]^{\varphi}   \ar@<0.4ex>[r]^{\alpha}&
    \cdot\ar@<0.4ex>[l]^{\beta}  \ar@<0.4ex>[r]^{\gamma} &
\cdot\ar@<0.4ex>[l]^{\delta} \ar@(dr,ur)[]_{\psi}   }\]
with relations
\[\alpha\beta = \varphi^p,\ \delta\gamma = \psi^q,\
\varphi\alpha = \alpha\gamma\delta ,\ \beta\varphi = \gamma\delta\beta,\
\psi\delta = \delta\beta\alpha,\ \gamma\psi = \beta\alpha\gamma.\]
Note that the algebras in (c) and (d) appear in Erdmann's list of algebras of quaternion type \cite{Er}.
\end{example}

Finally we give an application to Krull dimension three.
Again, let $\Lambda':=k[[x,y,u,v]]/(f(x,y)+uv)$.
For $w\in\mathfrak{S}_n$, we put
\[U_i^w:=(u,f_{w(1)}f_{w(2)}\cdots f_{w(n)})\subset\Lambda'\ \mbox{ and }\ M'_w:=\bigoplus_{i=1}^nU_i^w.\]
For a non-empty subset $I$ of $\{1,\cdots,n\}$, we put
\[U_I:=(u,\prod_{i\in I}f_i)\subset\Lambda'.\]
We have the following results from Theorems \ref{periodicity}, \ref{n! cluster tilting}, \ref{CT and NCCR} and \ref{derived equivalence}.
\begin{theorem}
Assume that $f_i\notin(x,y)^2$ for any $i$.
\begin{itemize}
\item[(a)] There exist exactly $n!$ basic $2$-cluster tilting objects $M'_w$ ($w\in{\mathfrak S}_n$) in $\CM(\Lambda')$.
\item[(b)] There exist exactly $(2^n-1)$ indecomposable rigid objects $U_I$ ($I\subset\{1,\cdots,n\}$, $I\neq\emptyset$) in $\CM(\Lambda')$.
\item[(c)] $\Lambda'$ has non-commutative crepant resolutions $\End_{\Lambda'}(M'_w)$ ($w\in{\mathfrak S}_n$),
which are mutually derived equivalent 3-CY algebras.
\end{itemize}
\end{theorem}

\end{document}